\documentclass{article}
\usepackage[utf8]{inputenc}
\usepackage{amsmath,mathtools,amsthm,amssymb, xcolor}
\usepackage{cite}
\usepackage{tikz}
\usepackage{comment}
\usepackage{abstract}
\usepackage{multicol}
\usepackage{tikz-cd}
\usetikzlibrary{decorations.pathreplacing,calligraphy,decorations, decorations.markings}
\pgfdeclaredecoration{simple line}{initial}{
	\state{initial}[width=\pgfdecoratedpathlength-1sp]{\pgfmoveto{\pgfpointorigin}}
	\state{final}{\pgflineto{\pgfpointorigin}}
}
\newcommand{\diam}{\mathord{\diamond}}

\usepackage{fullpage}

\usepackage[pdftex,hypertexnames=false,linktocpage=true]{hyperref}
\hypersetup{colorlinks=true,linkcolor=black,anchorcolor=blue,citecolor=black,filecolor=blue,urlcolor=blue,bookmarksnumbered=true,pdfview=FitB}

\newtheorem{theorem}{Theorem}[section]
\newtheorem{corollary}[theorem]{Corollary}
\newtheorem{lemma}[theorem]{Lemma}

\newtheorem{prop}[theorem]{Proposition}
\newtheorem{conjecture}[theorem]{Conjecture}
\newtheorem{observation}[theorem]{Observation}
\newtheorem{question}[theorem]{Question}

\theoremstyle{definition}
\newtheorem{definition}[theorem]{Definition}
\newtheorem{example}[theorem]{Example}
\newtheorem{remark}[theorem]{Remark}

\newcommand\N{\mathbb{N}}

\usepackage{graphicx}

\DeclarePairedDelimiter\set{\{}{\}}
\DeclarePairedDelimiter\abs{|}{|}
\newcommand{\red}{\text{red}}
\newcommand{\oiso}{~{\cong}_{\prec}~}
\newcommand{\iso}{\cong}

\title{Graph Universal Cycles: Compression and Connections to Universal Cycles}
\author{Rachel Kirsch\thanks{Department of Mathematical Sciences, George Mason University, Fairfax, VA.}
	\and 
	Clare Sibley\thanks{Houston, TX.}
	\and
	Elizabeth Sprangel\thanks{Department of Mathematics, Iowa State University, Ames, IA.}}

\begin{document}
	\maketitle
	\begin{abstract} Universal cycles, such as De Bruijn cycles, are cyclic sequences of symbols that represent every combinatorial object from some family exactly once as a consecutive subsequence. Graph universal cycles are a graph analogue of universal cycles introduced in 2010. We introduce graph universal partial cycles, a more compact representation of graph classes, which use ``do not know" edges. We show how to construct graph universal partial cycles for labeled graphs, threshold graphs, and permutation graphs. For threshold graphs and permutation graphs, we demonstrate that the graph universal cycles and graph universal partial cycles are closely related to universal cycles and compressed universal cycles, respectively. Using the same connection, for permutation graphs, we define and prove the existence of an $s$-overlap form of graph universal cycles. We also prove the existence of a generalized form of graph universal cycles for unlabeled graphs.
	\end{abstract}
	
	\section{Introduction}
	
	\subsection{Universal cycles}
	
	A \emph{universal cycle} for a family of combinatorial objects $\mathcal{F}$ is a cyclic sequence of symbols whose consecutive substrings of a fixed length represent each object in $\mathcal{F}$ exactly once. A canonical example is the \emph{De Bruijn cycle}: a universal cycle for the words of length $n$ over an alphabet $\mathcal{A}$. For example, $00010111$ is a De Bruijn cycle for words of length $3$ over the alphabet $\set{0,1}$ because its consecutive substrings of length $3$ are $000$, $001$, $010$, $101$, $011$, $111$, $110$, and $100$. De Bruijn cycles have proven to be relevant in a variety of settings such as robotics, computer science, game theory, bioinformatics, and materials science \cite{S01,BBS83,W07,AMM11,MMF21}.
	
	Universal cycles for other families of combinatorial objects---such as sets, permutations, and set partitions---were studied in \cite{CDG92}. In contrast to De Bruijn cycles, these universal cycles require a nontrivial encoding of the combinatorial objects into words.
	
	\subsection{Graph universal cycles}
	In 2010, Brockman, Kay, and Snively introduced graph universal cycles as an extension of universal cycles to list families of graphs and hypergraphs compactly \cite{BKS10}. Graph universal cycles are cyclically ordered graphs that contain each graph from some class as an induced subgraph on consecutive vertices exactly once. Although we can construct universal cycles for labeled graphs by encoding each graph as a binary word of length $\binom{n}{2}$---each bit indicating whether a pair of vertices induces an edge or not---graph universal cycles seem to be a more natural fit for representing graphs since no encoding of graphs into words is needed.
	
	Recently, in \cite{CGGP21}, Cantwell, Geraci, Godbole, and Padilla found graph universal cycles for sets, multisets, permutations, permutation involutions, and set partitions by encoding them as graphs. In Section \ref{sec:perms}, by defining graph universal cycles for permutations slightly differently, we uncover their deep connection to universal cycles for permutations.
	
	\begin{figure}[h!]
		\centering
		\tikzstyle{every node}=[circle, draw, fill=black, inner sep=0pt, minimum width=4pt]
		\begin{tikzpicture}[scale=0.4]
			\node[label={[shift={(0,-0.4)}]\scriptsize{1}}] (1) at (0,0) {};
			\node[label={[shift={(0,-0.4)}]\scriptsize{2}}] (2) at (2,0) {};
			\node[label={[shift={(0,-0.4)}]\scriptsize{3}}] (3) at (4,0) {};
			\node[label={[shift={(0,-0.4)}]\scriptsize{4}}] (4) at (6,0) {};
			\node[label={[shift={(0,-0.4)}]\scriptsize{5}}] (5) at (8,0) {};
			\node[label={[shift={(0,-0.4)}]\scriptsize{6}}] (6) at (10,0) {};  
			\node[label={[shift={(0,-0.4)}]\scriptsize{7}}] (7) at (12,0) {};  
			\node[label={[shift={(0,-0.4)}]\scriptsize{8}}] (8) at (14,0) {};
			
			\draw
			(3) to (4)
			(4) to (5)
			(5) to (6)
			(7) to (8)
			(1) to[bend left=55] (3)
			(2) to[bend left=55] (4)
			(4) to[bend left=55] (6)
			(5) to[bend left=55] (7);
			
		\end{tikzpicture}
		\caption{A graph universal cycle for the labeled graphs on vertex set $[3]$.}
		\label{fig:gucycle}
	\end{figure}
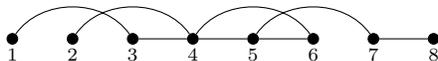
	
	\subsection{Compression of universal cycles}
	
	A significant reason for the widespread interest in and applicability of universal cycles is the efficiency gained from listing the combinatorial objects in a compact way. In the last few years, there has been a great deal of interest in shortening universal cycles even further. A key innovation is the use of a ``do not know'' symbol $\diamond$. In combinatorics on words, a partial word over an alphabet $\mathcal{A}$ is a sequence of symbols from $\mathcal{A}\cup\set{\diamond}$, where the alphabet is a set of symbols excluding the reserved character $\diam$; typically, $\mathcal{A} = \set{0,1,\ldots, \abs{\mathcal{A}}-1}$. We say that a partial word $v$ \emph{covers} a partial word $u$ if and only if $u$ appears consecutively in $v$ after possibly replacing some $\diam$'s in $v$ with letters. For example, $\diam1\diam$ covers $01\diam$, $11\diam$, $\diam10$, $\diam11$, $010$, $011$, $110$, and $111$, and so  $1\diam1\diam0$ does too. Where the ``do not know'' symbol $\diamond$ is used, more than one word is covered in the same position, which further shortens universal cycles.
	
	The idea of using the $\diam$ symbol in a version of universal cycles first appeared in \cite{BSSW10} in 2010, where the central objects of study, called De Bruijn partial words, cover each word at least once. Universal partial words and universal partial cycles are partial words and cyclic partial words, respectively, which cover each word exactly once. They were introduced in \cite{CKMS17} in 2017 and further studied in \cite{G18,FGKMM21}.
	
	Similarly, in \cite{KPV19,KLSS22}, universal cycles for permutations were shortened using incomparable elements at distance $n-1$ and using ``do not know'' symbols.
	
	We introduce graph universal partial cycles, using ``do not know'' symbols to shorten graph universal cycles.
	
	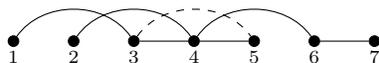
\begin{figure}[h!]
		\centering
		\tikzstyle{every node}=[circle, draw, fill=black, inner sep=0pt, minimum width=4pt]
		\begin{tikzpicture}[scale=0.4]
			\node[label={[shift={(0,-0.4)}]\scriptsize{1}}] (1) at (0,0) {};
			\node[label={[shift={(0,-0.4)}]\scriptsize{2}}] (2) at (2,0) {};
			\node[label={[shift={(0,-0.4)}]\scriptsize{3}}] (3) at (4,0) {};
			\node[label={[shift={(0,-0.4)}]\scriptsize{4}}] (4) at (6,0) {};
			\node[label={[shift={(0,-0.4)}]\scriptsize{5}}] (5) at (8,0) {};
			\node[label={[shift={(0,-0.4)}]\scriptsize{6}}] (6) at (10,0) {};  
			\node[label={[shift={(0,-0.4)}]\scriptsize{7}}] (7) at (12,0) {};  
			
			\draw
			(3) to (4)
			(4) to (5)
			(6) to (7)
			(1) to[bend left=55] (3)
			(2) to[bend left=55] (4)
			(4) to[bend left=55] (6);
			
			\draw[dashed]
			(3) to[bend left=55] (5);
			
		\end{tikzpicture}
		\caption{A graph universal partial cycle for the labeled graphs on vertex set $[3]$, where vertices $3,4$ and $5$ cover both the triangle and an ordered path.}
		\label{fig:gucycle}
	\end{figure}
	
	In Section \ref{sec:prelim} we formally establish preliminary definitions and results to be used throughout the paper. Section \ref{sec:labeled} constructs graph universal partial cycles for labeled graphs. In Section \ref{sec:unlabeled} we revise and prove a weaker form of a conjecture from \cite{BKS10} on graph universal cycles for unlabeled graphs, and show an example of a graph universal partial word for unlabeled graphs. Section \ref{sec:threshold} proves the existence of graph universal cycles for threshold graphs by way of a connection to binary De Bruijn cycles and extends this connection to compressed forms of each. In Section \ref{sec:perms} we show a connection between graph universal cycles for permutation graphs and universal cycles for permutations; extend the connection to compressed forms of each and $s$-overlap forms of each; and show examples of other types of graph universal partial cycles for permutation graphs and two new ways to shorten universal cycles for permutations. Finally, in Section \ref{sec:open} we summarize related open problems.
	
	\section{Preliminaries}\label{sec:prelim}
	
	In this paper, graphs are unlabeled unless otherwise specified. 
	
	\subsection{Ordered Graphs and Gucycles}
	\begin{definition}
		An \emph{ordered graph} is a graph whose vertices are totally ordered.
	\end{definition}
	
	Typically the total order on the vertices will be represented by labeling the vertices with $[N] := \set{1, 2, \ldots, N}$, where $N$ is the number of vertices. For an $N$-vertex ordered graph $G$, we write $\red(G)$ for the ordered graph resulting from (re)labeling $G$ so that its vertex set is $[N]$ while preserving the vertex order.
	
	\begin{definition}
		If $G$ and $H$ are $N$-vertex ordered graphs, we say $G$ and $H$ are \emph{isomorphic as ordered graphs} (or G is \emph{isomorphic as an ordered graph} to $H$), denoted by $G \oiso H$, if $\red(G) = \red(H)$, i.e., the unique bijection between $V(G)$ and $V(H)$ that is specified by their total orders is a graph isomorphism.
	\end{definition}
	
	\begin{definition}
		A \emph{cyclically ordered graph} is a graph whose vertices are cyclically ordered.
	\end{definition}
	
	Again we will use $[N]$ as a set of vertex labels. The cyclic order is the standard one: after vertex $x$, the next vertex is $x+1$, unless $x = N$, in which case the next vertex is $1$.
	
	\begin{definition} A \emph{$k$-window} (or \emph{window}) of a (cyclic) word is $k$ consecutive letters of the (cyclic) word. A \emph{$k$-window} (or \emph{window}) of a (cyclically) ordered graph is the induced ordered subgraph on $k$ consecutive vertices of $G$. 
	\end{definition}
	
	For a (cyclically) ordered graph $G$, we will also use the notation $G[i, i+1, \ldots, i+k-1]$ to refer to the $k$-window of $G$ starting at vertex $i$ (and wrapping around the cycle if needed and if $G$ is cyclically ordered).
	
	\begin{definition}
		We say an ordered graph $H$ is \emph{contained} in a (cyclically) ordered graph $G$ if $H$ is isomorphic as an ordered graph to a window of $G$.
	\end{definition}
	
	\begin{definition}\label{def:gucycle}
		A \emph{graph universal cycle} (or \emph{gucycle}) for a family of $n$-vertex graphs $\mathcal{F}$ is a cyclically ordered graph $G$ that contains each graph in $\mathcal{F}$ exactly once. A \emph{graph universal word} (or \emph{guword}) for a family of $n$-vertex graphs $\mathcal{F}$ is an ordered graph $G$ that contains each graph in $\mathcal{F}$ exactly once.
	\end{definition}
	
	Since any edge at (cyclic) distance at least $n$ in a gucycle or guword is irrelevant to the $n$-windows, for simplicity we assume gucycles and guwords have no such edges, unless otherwise specified.
	
	The number of vertices in a gucycle is always $\abs{\mathcal{F}}$, the number of graphs it represents, since there is exactly one graph in each $n$-window, and there is exactly one $n$-window starting at each vertex of the gucycle.
	
	\subsection{Partial Graphs and Gupcycles}\label{subsec:partial}
	
	We make the following definitions in order to create a compressed form of gucycles.
	
	\begin{definition}\label{def:partialgraph}
		A \emph{partial graph} $G$ is a vertex set $V(G)$ together with a partition of $\binom{V(G)}{2}$ into an edge set $E(G)$, a non-edge set $\bar{E}(G)$, and a diamond edge set $E_{\diamond}(G)$. 
	\end{definition}
	
	The diamond edges are ``wildcard'' pairs of vertices, which may be thought of as representing either an edge or a non-edge. The diamond edge set may be empty, so all graphs are partial graphs.
	
	\begin{definition}\label{def:partialgraphcontain}
		We say a graph $H$ is \emph{contained} in a partial graph $G$ if there is a mapping $f: V(H) \to V(G)$ that satisfies the following conditions:
		\begin{enumerate}
			\item If $ij \in E(H)$, then $f(i)f(j) \in E(G) \cup E_{\diamond}(G)$, and
			\item if $ij \notin E(H)$, then $f(i)f(j) \in \bar{E}(G) \cup E_{\diamond}(G)$.
		\end{enumerate}
		
		Similarly, we say an ordered graph $H$ is \emph{contained} in a (cyclically) ordered partial graph $G$ if $H$ is isomorphic as an ordered graph to a window of a (cyclically) ordered graph obtained from $G$ by regarding each of the diamond edges as either an edge or a non-edge, i.e. a (cyclically) ordered graph whose vertex set is $V(G)$ and whose edge set is $E(G)\cup S$, where $\emptyset \subseteq S \subseteq E_{\diamond}(G)$.
	\end{definition}

	\begin{definition}\label{def:gupcycle}
		A \emph{graph universal partial cycle}, or \emph{gupcycle}, for a family of $n$-vertex graphs $\mathcal{F}$ is a cyclically ordered partial graph $G$ that contains each graph in $\mathcal{F}$ exactly once. A \emph{graph universal partial word}, or \emph{gupword}, for a family of $n$-vertex graphs $\mathcal{F}$ is an ordered partial graph $G$ that contains each graph in $\mathcal{F}$ exactly once.
	\end{definition}
	
	As in the definition of gucycles (Definition \ref{def:gucycle}), we assume that gupcycles and gupwords have no edges that are irrelevant to the $n$-windows, unless otherwise specified. 
	
	\begin{example}
		Figure \ref{fig:cyclicgluing} shows an ordered partial graph on the left and a cyclically ordered partial graph on the right, which are a graph universal partial word and a graph universal partial cycle, respectively, for the labeled graphs on vertex set $[3]$. The diamond edges are drawn with dashed arcs. The $i$th $3$-window of the gupword is the same as the $i$th $3$-window of the gupcycle for every $i$. The second, third, and fourth $3$-windows contain two ordered graphs each. For example, the third window contains both a triangle $\set{12, 23, 13}$ and a path $\set{12, 23}$.

		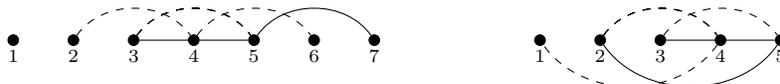
\begin{figure}[h!]
			\centering
			\tikzstyle{every node}=[circle, draw, fill=black, inner sep=0pt, minimum width=4pt]
			\begin{tikzpicture}[scale=0.4]
				\node[label={[shift={(0,-0.4)}]\scriptsize{1}}] (0) at (-2,0) {};
				\node[label={[shift={(0,-0.4)}]\scriptsize{2}}] (1) at (0,0) {};
				\node[label={[shift={(0,-0.4)}]\scriptsize{3}}] (2) at (2,0) {};
				\node[label={[shift={(0,-0.4)}]\scriptsize{4}}] (3) at (4,0) {};
				\node[label={[shift={(0,-0.4)}]\scriptsize{5}}] (4) at (6,0) {};
				\node[label={[shift={(0,-0.4)}]\scriptsize{6}}] (5) at (8,0) {};
				\node[label={[shift={(0,-0.4)}]\scriptsize{7}}] (6) at (10,0) {};
				
				\draw
				(3) to (4)
				(2) to (3)
				(4) to[bend left=55] (6);
				
				\draw[dashed]
				(1) to[bend left=55] (3)
				(2) to[bend left=55] (4)
				(3) to[bend left=55] (5)
				(2) to[bend left=55] (4);

				\draw[dashed, color=white]
				(4) to[bend left=55] (1);
				
			\end{tikzpicture}
			\hspace{5em}
			\begin{tikzpicture}[scale=0.4]
				\node[label={[shift={(0,-0.4)}]\scriptsize{1}}] (1) at (0,0) {};
				\node[label={[shift={(0,-0.4)}]\scriptsize{2}}] (2) at (2,0) {};
				\node[label={[shift={(0,-0.4)}]\scriptsize{3}}] (3) at (4,0) {};
				\node[label={[shift={(0,-0.4)}]\scriptsize{4}}] (4) at (6,0) {};
				\node[label={[shift={(0,-0.4)}]\scriptsize{5}}] (5) at (8,0) {};
				
				\draw
				(5) to[bend left=55] (2)
				(3) to (4)
				(4) to (5);
				
				\draw[dashed]
				(2) to[bend left=55] (4)
				(3) to[bend left=55] (5)
				(4) to[bend left=55] (1)
				(2) to[bend left=55] (4);
				
			\end{tikzpicture}
			\caption{A gupword and a gupcycle for the labeled graphs on vertex set $[3]$, which have the same $3$-windows.}
			\label{fig:cyclicgluing}
		\end{figure}
		
	\end{example}
	
	We will also find it useful to extend the definition of isomorphism to partial graphs.
	
	\begin{definition}
		If $G$ and $H$ are $N$-vertex ordered partial graphs, we say $G$ and $H$ are \emph{isomorphic as ordered partial graphs} (or G is \emph{isomorphic as an ordered partial graph} to $H$), denoted by $G \oiso H$, if $\red(G) = \red(H)$, i.e., the unique bijection between $V(G)$ and $V(H)$ that is specified by their total orders preserves the edges, non-edges, and diamond edges.
	\end{definition}

	\subsection{Twins and Compression}
	
	In Sections \ref{sec:labeled}, \ref{sec:perms}, and \ref{sec:open} we will use the idea of compressing twin edges. \emph{Twin edges}, also called multiple edges or parallel edges, are directed edges having the same head vertices and the same tail vertices. We are sometimes able to \emph{compress} a pair of twin edges $e$ and $f$ by replacing them with a single edge $g$ that has the same head and tail vertices as $e$ and $f$, and that is labeled in such a way that it represents the labels of $e$ and $f$. As slightly different definitions of compression are needed in different settings, we will define it more precisely in Sections \ref{sec:labeled}, \ref{sec:perms}, and \ref{sec:open}. After compressing all pairs of twin edges in a directed multigraph, we have \emph{fully compressed} it, while a \emph{partially compressed} directed multigraph is the result of compressing some but not necessarily all of the twin edges.
	
	\subsection{Tours, the Arc Digraph, Overlaps, and Gluing}
	\begin{definition}
		A \emph{tour} in a (directed) graph $G$ is a closed walk with no repeated edges. An \emph{Euler tour} is a tour that covers every edge in $G$ exactly once. A (directed) graph is \emph{Eulerian} if it has an Euler tour. A directed graph $G$ is \emph{balanced} if every vertex has in-degree equal to its out-degree, and \emph{strongly connected} if, for every pair of distinct vertices $u$ and $v$, there is a directed path from $u$ to $v$.
	\end{definition}
	
	We will use the following standard fact.
	
	\begin{prop}
		Every balanced and strongly connected directed graph is Eulerian.
	\end{prop}
	
	The De Bruijn cycles for words of length $n$ over an alphabet $\mathcal{A}$ correspond bijectively to the Euler tours of the De Bruijn graph, a directed graph that encodes the overlaps between words. The number of De Bruijn cycles was obtained in 1951 by counting these Euler tours.
	
	\begin{theorem}[Van Aardenne-Ehrenfest and De Bruijn \cite{AB51}, Smith and Tutte \cite{TS41}]\label{thm:DBcount}
		The number of De Bruijn cycles for $\mathcal{A}^n$, where $a = \abs{\mathcal{A}}$, is $\frac{(a!)^{a^{n-1}}}{a^{n}}$.
	\end{theorem}
	
	To work with graph universal (partial) cycles we will need the following analogous definitions of overlaps and a directed multigraph that encodes them, which we call the arc digraph, following \cite{BKS10, CGGP21}.
	
	\begin{definition}\label{def:partialgraphoverlap}
		Let $G$ be a (partial) graph on vertex set $[m]$ and $H$ be a (partial) graph on vertex set $[n]$. For $1 \le s \le \min\set{m,n}-1$, the ordered pair $(G,H)$ \emph{overlaps by $s$} if $G[m-s+1, \ldots, m] \oiso H[1, \ldots, s]$, i.e., the $s$-windows at the end of $G$ and beginning of $H$ are isomorphic as ordered (partial) graphs. In the case $m=n$, the ordered pair $(G,H)$ \emph{overlaps} if it overlaps by $n-1$.
	\end{definition}
	
	\begin{definition}
		The \emph{arc digraph for labeled graphs on $[n]$} (or \emph{arc digraph}), denoted by $D_n$, is the directed multigraph whose vertices are the labeled graphs on vertex set $[n-1]$ and whose edge set is the set of labeled graphs on vertex set $[n]$. An edge $e$ points from $u$ to $v$ if and only if $e[1,2,\ldots,n-1] \oiso u$ and $e[2,3,\ldots,n]\oiso v$.
	\end{definition}
	
	We now define the operation of gluing two overlapping graphs in order to construct gucycles and gupcycles from Euler tours.
	
	\begin{lemma}[Gluing Lemma]\label{lem:gluing}
		Let $G$ be a (partial) graph on vertex set $[m]$ and $H$ be a (partial) graph on vertex set $[n]$. Then the pair $(G,H)$ overlaps by $s$ if and only if there exists a (partial) graph $J$ on vertex set $[m+n-s]$ such that $J[1,2,\ldots,m] \oiso G$ and $J[m-s+1,m-s+2,\ldots,m-s+n] \oiso H$.
	\end{lemma}
	
	\begin{proof}Given such a (partial) graph $J$, its induced subgraphs $G:=J[1,2,\ldots,m]$ and $H:=J[m-s+1,m-s+2,\ldots,m-s+n]$ must overlap by $s$ since $G[m-s+1,m-s+2,\ldots,m]\oiso J[m-s+1,m-s+1,\ldots,m] \oiso H[1,2,\ldots,s]$. For the opposite direction, a (partial) graph $J$ can be constructed by copying $G$ to $J[1,2,\ldots,m]$ and copying $H$ to $J[m-s+1,m-s+2,\ldots,m-s+n]$, which creates no conflicts since $(G,H)$ overlaps by $s$.
	\end{proof}
	
	In the Gluing Lemma, note that the presence or absence of any edges (or diamond edges) of $J$ with one vertex in $\set{1,\ldots,m-s}$ and the other in $\set{m+1,\ldots,m+n-s}$ is irrelevant, so multiple such (partial) graphs $J$ exist. (For example, any pair of twin edges from $G$ to $H$ in $D_n$ represents two different such graphs $J$; see Observations \ref{obs:twinedge} and \ref{obs:twinexistence}.) Lemma \ref{lem:gluing} makes possible the following definition of gluing.
	
	\begin{definition}
		If $G$ is a (partial) graph on vertex set $[m]$, $H$ is a (partial) graph on vertex set $[n]$, and $(G, H)$ overlaps by $s$, the \emph{gluing of $(G,H)$} is the unique (partial) graph $J$ on vertex set $[m+n-s]$ such that $J[1,2,\ldots,m] \oiso G$ and $J[m-s+1,m-s+2,\ldots,m-s+n] \oiso H$, and that has no edges (or diamond edges) with one vertex in $\set{1,\ldots,m-s}$ and the other in $\set{m+1,\ldots,m+n-s}$.
	\end{definition}
	
	To construct graph universal (partial) cycles we will also need to be able to glue together the beginning and end of an ordered (partial) graph, creating a cyclically ordered (partial) graph.
	
	\begin{lemma}[Cyclic Gluing Lemma]\label{lem:cyclicgluing}
		Let $G$ be a (partial) graph on vertex set $[m]$. Suppose $G$ has no edges (and no diamond edges) with one vertex in $\set{1,2,\ldots, n-1}$ and the other in $\set{m-n+2,m-n+3,\ldots,m}$, and no edges (or diamond edges) between vertices at distance greater than $n-1$. If the pair $(G,G)$ overlaps then there exists a cyclically ordered (partial) graph $J$ on vertex set $[m-n+1]$ such that $G$ as an ordered (partial) graph and $J$ as a cyclically ordered (partial) graph have the same $n$-windows, in the same order, for all $n$ satisfying $m-n+1 > 2(n-1)$.
		
		Conversely, let $J$ be a cyclically ordered (partial) graph $J$. Let $L$ be its length, let $n$ satisfy $L > 2(n-1)$, and suppose $J$ has no edges (or diamond edges) between vertices at distance greater than $n-1$. Then there exists an ordered (partial) graph $G$ having the same $n$-windows as $J$.
	\end{lemma}
	
	\begin{proof}Let $G$ be such an ordered (partial) graph on vertex set $[m]$ for which $(G,G)$ overlaps. Identify the vertices $\set{1,2,\ldots, n-1}$ with the vertices $\set{m-n+2,m-n+3,\ldots,m}$, in order, to form a cyclically ordered (partial) graph $J$ on vertex set $[m-n+1]$. In other words, for each vertex $v \in \set{m-n+2,m-n+3,\ldots,m}$, replace all edges (and diamond edges) $\set{v,w}$, where $w \le m-n+1$, with $\set{v-m+n-1,w}$; then delete $v$ and its other edges (and diamond edges). By hypothesis, $w \ge n$. The edges being added have the form $\set{v-m+n-1,w}$, where $1 \le v-m+n-1 \le n-1$. 
		Since every edge (or diamond edge) $\set{v,w}$ in $G$ is at distance at most $n-1$, we have $v-w\le n-1$, so the distance from $v-m+n-1$ to $w$ is $w-(v-m+n-1)\ge m-2(n-1)>n-1$, i.e., $\set{v-m+n-1,w}$ is not in any $n$-window of the form $J[i, i+1, \ldots, i+n-1]$ for $i \in [m-2(n-1)]$. Therefore we have $J[i, i+1, \ldots, i+n-1] = G[i, i+1, \ldots, i+n-1]$ for every $i \in [m-2(n-1)]$. The other $n$-windows of $J$ have the form $J[i, \ldots, m-n+1, 1, \ldots, 2(n-1)-m+i]$ for $i \in [m-n+1]\setminus[m-2(n-1)]$. By construction such windows are isomorphic as ordered (partial) graphs to $G[i, \ldots, m-n+1, m-n+2, \ldots, i+n-1]$ for every $i \in [m-n+1]\setminus[m-2(n-1)]$.
		
		For the other direction, let $J$ be a cyclically ordered (partial) graph on vertex set $[L]$. For $n$ satisfying $L > 2(n-1)$, let $m = L+n-1$. Construct an ordered (partial) graph $G$ on vertex set $[m]$ as follows. First, let $G$ have the same edges and non-edges (and diamond edges) as $J$ on $[L]$, and no edges (or diamond edges) on $[m]\setminus[L]$. For every edge (or diamond edge) $\set{u,w}$ of $J$ with $1 \le u \le n-1$ and $n \le w \le L$, either the distance from $u$ to $w$ is at most $n-1$, or the distance from $w$ to $u$ is at most $n-1$, but not both (because $L > 2(n-1)$) or neither (by hypothesis). In $G$, keep $\set{u,w}$ if the distance from $u$ to $w$ is at most $n-1$, and replace it with $\set{u+L,w}$ otherwise. As a result we have $G[i, i+1, \ldots, i+n-1] = J[i, i+1, \ldots, i+n-1]$ for every $i \in [L-n+1]$, and $G[i, \ldots, L, L+1, \ldots, i+n-1] \oiso J[i, i+1, \ldots, L, 1, \ldots, n-L+i-1]$ for every $i \in [m-n+1]\setminus[L-n+1]$.
	\end{proof}
	
	We will also need the following generalization of the Cyclic Gluing Lemma.
	
	\begin{lemma}[$s$-Overlap Cyclic Gluing Lemma]\label{lem:scyclicgluing} Let $n$, $s$, $m$, and $k$ be positive integers satisfying $1 \le s\le n-1$, $m=n+k(n-s)$, and $m-s>2(n-1)$.
		Let $G$ be a (partial) graph on vertex set $[m]$. Suppose $G$ has no edges (and no diamond edges) with one vertex in $\set{1,2,\ldots, s}$ and the other in $\set{m-s+1,m-s+2,\ldots,m}$, and every edge (or diamond edge) is contained in an $n$-window of the form $G[1+i(n-s), \ldots, n+i(n-s)]$ for some $0\le i \le k$. If the pair $(G,G)$ overlaps by $s$ then there exists a cyclically ordered (partial) graph $J$ on vertex set $[m-s]$ such that $G[1+i(n-s), \ldots, n+i(n-s)] \oiso J[1+i(n-s), \ldots, n+i(n-s)]$ for all $0\le i \le k$.
	\end{lemma}
	
	It is helpful to notice that $m = n+k(n-s) = s + (k+1)(n-s)$.
	
	\begin{proof}
		Let $G$ be an ordered (partial) graph on $[m]$ for which $(G,G)$ overlaps by $s$. Identify the vertices $\{1,2,\dots,s\}$ with the vertices $\{m-s+1,m-s+2\dots, m\}$, in order, to form a cyclically ordered (partial) graph $J$ on vertex set $[m-s]$. In other words, for each vertex $v\in \{m-s+1,m-s+2,\dots,m\}$, replace all edges (and diamond edges) $\{v,w\}$, where $w\le m-s$, with $\{v-m+s,w\}$; then, delete $v$ and its other edges (and diamond edges). In $G$ there are no edges (or diamond edges) at distance at least $n$, so $v-w\le n-1$. Therefore, the distance from $v-m+s$ to $w$ is $w-(v-m+s)\ge -(n-1)+(m-s) >n-1$. That is, $\{v-m+s, w\}$ is not in any $n$-window of the form $J[i,i+1,\dots,i+n-1]$ for any $i\in[m-s-(n-1)]$. Therefore, $G[1+i(n-s), \ldots, n+i(n-s)] \oiso J[1+i(n-s), \ldots, n+i(n-s)]$ whenever $n+i(n-s)\le m-s$, or all $i\le k+1-\frac{n}{n-s}$. The other $n$-windows under consideration are of the form $J[1+i(n-s),2+i(n-s),\dots,m-s,1,2,\dots,n + (i-k-1)(n-s)]$ for $k+1-\frac{n}{n-s}< i\le k$. By construction, these windows are isomorphic as ordered (partial) graphs to $G[1+i(n-s),2+i(n-s),\dots,m-s, m-s+1, m-s+2, \dots,n+i(n-s)]$.
	\end{proof}
	
	\begin{definition}
		If the ordered (partial) graph $G$ satisfies the conditions of Lemma \ref{lem:cyclicgluing} or Lemma \ref{lem:scyclicgluing} then the \emph{cyclic gluing} of $G$ is the cyclically ordered (partial) graph $J$ constructed in the proof of the lemma.
	\end{definition}
	
	By definition, the cyclic gluing of $G$ has the same $n$-windows as $G$ and in the same order, and it has no edges (or diamond edges) that are irrelevant to the $n$-windows.
	
	By Lemmas \ref{lem:gluing} and \ref{lem:cyclicgluing}, an Euler tour in the arc digraph defines a gucycle for labeled graphs on vertex set $[n]$: given an Euler tour one can successively glue together the consecutive graphs, and provided the tour is long enough relative to $n$, the ends can be glued together to form a gucycle; also given a gucycle one can take its $n$-windows to form an Euler tour. Similarly an Euler tour of a partially compressed arc digraph defines a gupcycle provided that it is long enough. We will use the arc digraph to construct gupcycles for labeled graphs in Section \ref{sec:labeled} and to construct $f$-fold gucycles for unlabeled graphs in Section \ref{sec:ffold}. In Section \ref{sec:perms} we use a subgraph of $D_n$ that includes only the permutation graphs.
	
	The following table shows the construction of the gupword for labeled graphs on vertex set $[3]$ from Figure \ref{fig:cyclicgluing}. The first row shows each of its $3$-windows in order, corresponding to an Euler tour of a partially compressed arc digraph. Each ordered partial graph in the second row is the result of gluing the ordered partial graph to its left and the ordered partial graph above it. The resulting gupword can be cyclically glued with an overlap of $2$ to produce the gupcycle also shown in Figure \ref{fig:cyclicgluing}.
	
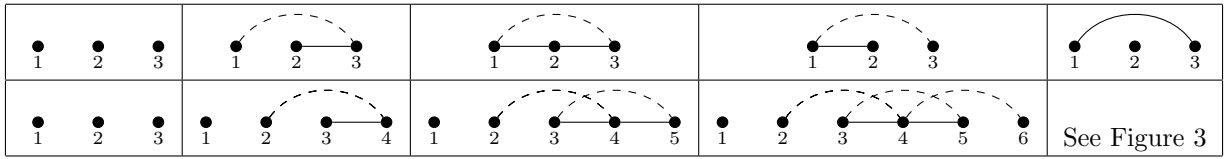
	\begin{center}
		\noindent\begin{tabular}{|c|c|c|c|c|}\hline
			\tikzstyle{every node}=[circle, draw, fill=black, inner sep=0pt, minimum width=4pt]
			\begin{tikzpicture}[scale=0.4]

				\node[label={[shift={(0,-0.4)}]\scriptsize{1}}] (1) at (0,0) {};
				\node[label={[shift={(0,-0.4)}]\scriptsize{2}}] (2) at (2,0) {};
				\node[label={[shift={(0,-0.4)}]\scriptsize{3}}] (3) at (4,0) {};

			\end{tikzpicture} & \tikzstyle{every node}=[circle, draw, fill=black, inner sep=0pt, minimum width=4pt]\begin{tikzpicture}[scale=0.4]

				\node[label={[shift={(0,-0.4)}]\scriptsize{1}}] (1) at (0,0) {};
				\node[label={[shift={(0,-0.4)}]\scriptsize{2}}] (2) at (2,0) {};
				\node[label={[shift={(0,-0.4)}]\scriptsize{3}}] (3) at (4,0) {};
				
				\draw
				(2) to (3);
				
				\draw[dashed]
				(1) to[bend left=55] (3);
				
			\end{tikzpicture} 
			& \tikzstyle{every node}=[circle, draw, fill=black, inner sep=0pt, minimum width=4pt]\begin{tikzpicture}[scale=0.4]

				\node[label={[shift={(0,-0.4)}]\scriptsize{1}}] (1) at (0,0) {};
				\node[label={[shift={(0,-0.4)}]\scriptsize{2}}] (2) at (2,0) {};
				\node[label={[shift={(0,-0.4)}]\scriptsize{3}}] (3) at (4,0) {};
				
				\draw
				(1) to (2)
				(2) to (3);
				
				\draw[dashed]
				(1) to[bend left=55] (3);
				
			\end{tikzpicture} & \tikzstyle{every node}=[circle, draw, fill=black, inner sep=0pt, minimum width=4pt]\begin{tikzpicture}[scale=0.4]

				\node[label={[shift={(0,-0.4)}]\scriptsize{1}}] (1) at (0,0) {};
				\node[label={[shift={(0,-0.4)}]\scriptsize{2}}] (2) at (2,0) {};
				\node[label={[shift={(0,-0.4)}]\scriptsize{3}}] (3) at (4,0) {};
				
				\draw (1) to (2);
				
				\draw[dashed]
				(1) to[bend left=55] (3);
				
			\end{tikzpicture} & \tikzstyle{every node}=[circle, draw, fill=black, inner sep=0pt, minimum width=4pt]\begin{tikzpicture}[scale=0.4]

				\node[label={[shift={(0,-0.4)}]\scriptsize{1}}] (1) at (0,0) {};
				\node[label={[shift={(0,-0.4)}]\scriptsize{2}}] (2) at (2,0) {};
				\node[label={[shift={(0,-0.4)}]\scriptsize{3}}] (3) at (4,0) {};
				
				\draw
				(1) to[bend left=55] (3);
				
			\end{tikzpicture}\\\hline
			\tikzstyle{every node}=[circle, draw, fill=black, inner sep=0pt, minimum width=4pt]
			\begin{tikzpicture}[scale=0.4]

				\node[label={[shift={(0,-0.4)}]\scriptsize{1}}] (1) at (0,0) {};
				\node[label={[shift={(0,-0.4)}]\scriptsize{2}}] (2) at (2,0) {};
				\node[label={[shift={(0,-0.4)}]\scriptsize{3}}] (3) at (4,0) {};

			\end{tikzpicture}
			& 
			\tikzstyle{every node}=[circle, draw, fill=black, inner sep=0pt, minimum width=4pt]\begin{tikzpicture}[scale=0.4]
				
				\node[label={[shift={(0,-0.4)}]\scriptsize{1}}] (1) at (0,0) {};
				\node[label={[shift={(0,-0.4)}]\scriptsize{2}}] (2) at (2,0) {};
				\node[label={[shift={(0,-0.4)}]\scriptsize{3}}] (3) at (4,0) {};
				\node[label={[shift={(0,-0.4)}]\scriptsize{4}}] (4) at (6,0) {};
				
				\draw(3) to (4);
				
				\draw[dashed]
				(2) to[bend left=55] (4)
				(2) to[bend left=55] (4);
				
			\end{tikzpicture} & \tikzstyle{every node}=[circle, draw, fill=black, inner sep=0pt, minimum width=4pt]\begin{tikzpicture}[scale=0.4]
				
				\node[label={[shift={(0,-0.4)}]\scriptsize{1}}] (1) at (0,0) {};
				\node[label={[shift={(0,-0.4)}]\scriptsize{2}}] (2) at (2,0) {};
				\node[label={[shift={(0,-0.4)}]\scriptsize{3}}] (3) at (4,0) {};
				\node[label={[shift={(0,-0.4)}]\scriptsize{4}}] (4) at (6,0) {};
				\node[label={[shift={(0,-0.4)}]\scriptsize{5}}] (5) at (8,0) {};
				
				\draw
				(3) to (4)
				(4) to (5);
				
				\draw[dashed]
				(2) to[bend left=55] (4)
				(3) to[bend left=55] (5)
				(2) to[bend left=55] (4);
				
			\end{tikzpicture} & \tikzstyle{every node}=[circle, draw, fill=black, inner sep=0pt, minimum width=4pt]\begin{tikzpicture}[scale=0.4]
				
				\node[label={[shift={(0,-0.4)}]\scriptsize{1}}] (1) at (0,0) {};
				\node[label={[shift={(0,-0.4)}]\scriptsize{2}}] (2) at (2,0) {};
				\node[label={[shift={(0,-0.4)}]\scriptsize{3}}] (3) at (4,0) {};
				\node[label={[shift={(0,-0.4)}]\scriptsize{4}}] (4) at (6,0) {};
				\node[label={[shift={(0,-0.4)}]\scriptsize{5}}] (5) at (8,0) {};
				\node[label={[shift={(0,-0.4)}]\scriptsize{6}}] (6) at (10,0) {};
				
				\draw
				(3) to (4)
				(4) to (5);
				
				\draw[dashed]
				(2) to[bend left=55] (4)
				(3) to[bend left=55] (5)
				(2) to[bend left=55] (4)
				(4) to[bend left=55] (6);
				
			\end{tikzpicture} & See Figure \ref{fig:cyclicgluing}\\\hline
		\end{tabular}
	\end{center}
	
	\section{Labeled Graphs}\label{sec:labeled}
	
	In \cite{BKS10} Brockman, Kay, and Snively proved the existence of gucycles for labeled graphs on $[n]$ for all $n \ge 3$. We introduce and prove the existence of graph universal partial cycles (gupcycles) for labeled graphs on $[n]$.
	
	Recall that the Euler tours of the arc digraph $D_n$ correspond to the gucycles for labeled graphs on $[n]$. We are interested in shortening such gucycles through the use of diamond edges. To accomplish this, we note that the edges of $D_n$ occur in pairs, which we call twins.
	
	\begin{definition}
		Two labeled graphs on $[n]$ are \emph{twins} if they are twin edges of $D_n$, i.e., 
		their induced subgraphs on vertex set $\set{1,2,\ldots,n-1}$ are the same, and their induced subgraphs on vertex set $\set{2,3,\ldots, n}$ are the same.
	\end{definition}
	
	\begin{observation}\label{obs:twinedge}
		Twin labeled graphs on $[n]$ differ only in the presence or absence of the edge $\set{1, n}$.
	\end{observation}
	
	\begin{observation}\label{obs:twinexistence}
		Every labeled graph on $[n]$ has a unique twin.
	\end{observation}
	
	It is useful to consider a compressed form of the arc digraph $D_n$ that has no twin edges. Given Observations \ref{obs:twinedge} and \ref{obs:twinexistence}, we lose no information in the compression.
	
	\begin{definition}
		We say that we \emph{compress a pair of twin edges} $e$ and $f$ if we replace the twin edges $e$ and $f$ in $G$ with a single edge $g$ that is a labeled partial graph on $[n]$ with $\set{1,n}$ as its sole diamond edge. The other pairs in $\binom{[n]}{2}$ are either edges or non-edges of $g$, consistent with both $e$ and $f$.
	\end{definition}
	
	\begin{definition}
		The \emph{fully compressed arc digraph} for labeled graphs on $[n]$, denoted by $D_n^*$, is the result of compressing every pair of twin edges in $D_n$.
	\end{definition}
	
	Alternatively, we can selectively compress some subset of twin pairs to construct gupcycles of different lengths, as seen in the following example.
	\begin{example}\label{ex:twincompression} We can compress along the twins of Figure \ref{fig:arcdigraph}.
		
		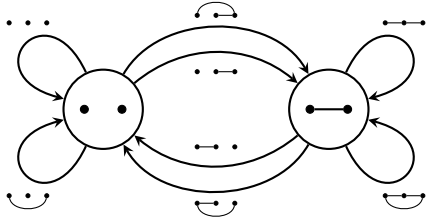
\begin{figure}[h!]
			\centering
			\tikzstyle{every node}=[circle, draw, fill=black, inner sep=0pt, minimum width=3pt]
			\begin{tikzpicture}[]
				
				\node[circle, draw, thick, fill=white, inner sep=0pt, minimum width=30pt] (0) at (.25,0) {};
				\node[circle, draw, thick, fill=white, inner sep=0pt, minimum width=30pt] (00) at (3.25,0) {};

				\draw[->, -stealth, thick] (00) to [bend right=-40] (0);
				\draw[->, -stealth, thick, bend left=60]
				(00) to (0);
				\draw[->, -stealth, thick, bend left=40]
				(0) to (00);
				\draw[->, -stealth, thick, bend left=60]
				(0) to (00);

				\draw[thick, ->, -stealth] (0) to [out=115,in=165,looseness=7] (0);
				\draw[thick, ->, -stealth] (0) to [out=245,in=195,looseness=7] (0);
				
				\draw[thick, ->, -stealth] (00) to [out=65,in=15,looseness=7] (00);
				\draw[thick, ->, -stealth] (00) to [out=-65,in=-15,looseness=7] (00);
				
				\node (1) at (0,0) {};
				\node (2) at (.5,0) {};
				\node (3) at (3, 0) {};
				\node (4) at (3.5,0) {};
				
				\draw[thick]
				(3) to (4);
				
				\node[circle, draw, fill=black, inner sep=0pt, minimum width=1.5pt] (a) at (1.75, 1.25) {};
				\node[circle, draw, fill=black, inner sep=0pt, minimum width=1.5pt] (b) at (1.5, 1.25) {};
				\node[circle, draw, fill=black, inner sep=0pt, minimum width=1.5pt] (c) at (2, 1.25) {};

				\draw (a) to (c);
				\draw(b) to [bend left=80] (c);

				\node[circle, draw, fill=black, inner sep=0pt, minimum width=1.5pt] (k) at (1.75, .5) {};
				\node[circle, draw, fill=black, inner sep=0pt, minimum width=1.5pt] (l) at (1.5, .5) {};
				\node[circle, draw, fill=black, inner sep=0pt, minimum width=1.5pt] (m) at (2, .5) {};
				
				\draw (k) to (m); 
				
				\node[circle, draw, fill=black, inner sep=0pt, minimum width=1.5pt] (d) at (1.75, -1.25) {};
				\node[circle, draw, fill=black, inner sep=0pt, minimum width=1.5pt] (e) at (1.5, -1.25) {};
				\node[circle, draw, fill=black, inner sep=0pt, minimum width=1.5pt] (f) at (2, -1.25) {};
				
				\draw (d) to (e);
				\draw(e) to [bend left=-80] (f);

				\node[circle, draw, fill=black, inner sep=0pt, minimum width=1.5pt] (h) at (1.75, -.5) {};
				\node[circle, draw, fill=black, inner sep=0pt, minimum width=1.5pt] (i) at (1.5, -.5) {};
				\node[circle, draw, fill=black, inner sep=0pt, minimum width=1.5pt] (j) at (2, -.5) {};
				
				\draw (h) to (i);

				\node[circle, draw, fill=black, inner sep=0pt, minimum width=1.5pt] (n) at (4.25, -1.15) {};
				\node[circle, draw, fill=black, inner sep=0pt, minimum width=1.5pt] (o) at (4, -1.15) {};
				\node[circle, draw, fill=black, inner sep=0pt, minimum width=1.5pt] (p) at (4.5, -1.15) {};
				
				\draw 
				(n) to (p)
				(n) to (o);
				\draw(o) to [bend left=-80] (p);

				\node[circle, draw, fill=black, inner sep=0pt, minimum width=1.5pt] (q) at (-.75, -1.15) {};
				\node[circle, draw, fill=black, inner sep=0pt, minimum width=1.5pt] (r) at (-.5, -1.15) {};
				\node[circle, draw, fill=black, inner sep=0pt, minimum width=1.5pt] (s) at (-1, -1.15) {};
				
				\draw(r) to [bend left=80] (s);

				\node[circle, draw, fill=black, inner sep=0pt, minimum width=1.5pt] (t) at (-.75, 1.15) {};
				\node[circle, draw, fill=black, inner sep=0pt, minimum width=1.5pt] (u) at (-.5, 1.15) {};
				\node[circle, draw, fill=black, inner sep=0pt, minimum width=1.5pt] (v) at (-1, 1.15) {};

				\node[circle, draw, fill=black, inner sep=0pt, minimum width=1.5pt] (w) at (4.25, 1.15) {};
				\node[circle, draw, fill=black, inner sep=0pt, minimum width=1.5pt] (x) at (4, 1.15) {};
				\node[circle, draw, fill=black, inner sep=0pt, minimum width=1.5pt] (y) at (4.5, 1.15) {};
				
				\draw (w) to (x)
				(w) to (y);
				
			\end{tikzpicture}  
			\caption{The arc digraph for labeled graphs on $[3]$.}
			\label{fig:arcdigraph}
		\end{figure}

		If we compress the twins along the $2$-cycle and one of the loops, notice that the resulting graph is still Eulerian. An Euler tour of this partially compressed arc digraph corresponds to a gupcycle for labeled graphs on $[3]$ containing three diamond edges. The gupcycle requires only $5$ vertices to account for all labeled graphs on $[3]$ exactly once.
		
		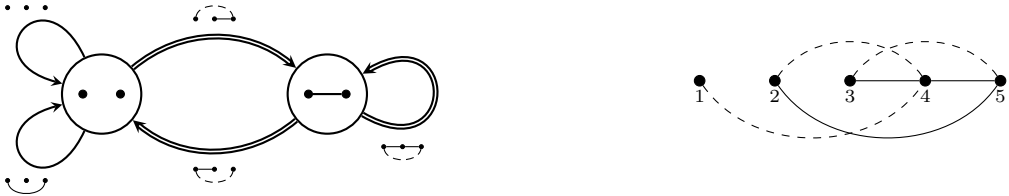
\begin{figure}[h!]
			\centering
			\begin{multicols}{2}
				\tikzstyle{every node}=[circle, draw, fill=black, inner sep=0pt, minimum width=3pt]
				\begin{tikzpicture}[]
					
					\node[circle, draw, thick, fill=white, inner sep=0pt, minimum width=30pt] (0) at (.25,0) {};
					\node[circle, draw, thick, fill=white, inner sep=0pt, minimum width=30pt] (00) at (3.25,0) {};

					\draw[->, -stealth, thick, double] (0) to [bend right=-40] (00);
					\draw[->, -stealth, thick, double, bend left=40]
					(00) to (0);

					\draw[thick, ->, -stealth] (0) to [out=115,in=165,looseness=7] (0);
					\draw[thick, ->, -stealth] (0) to [out=245,in=195,looseness=7] (0);
					
					\draw[thick, ->, -stealth,double] (00) to [out=-30,in=30,looseness=7] (00);
					
					\node (1) at (0,0) {};
					\node (2) at (.5,0) {};
					\node (3) at (3, 0) {};
					\node (4) at (3.5,0) {};
					
					\draw[thick]
					(3) to (4);

					\node[circle, draw, fill=black, inner sep=0pt, minimum width=1.5pt] (a) at (1.75, 1) {};
					\node[circle, draw, fill=black, inner sep=0pt, minimum width=1.5pt] (b) at (1.5, 1) {};
					\node[circle, draw, fill=black, inner sep=0pt, minimum width=1.5pt] (c) at (2, 1) {};

					\draw (a) to (c);
					\draw[densely dashed] (b) to [bend left=80] (c);

					\node[circle, draw, fill=black, inner sep=0pt, minimum width=1.5pt] (d) at (1.75, -1) {};
					\node[circle, draw, fill=black, inner sep=0pt, minimum width=1.5pt] (e) at (1.5, -1) {};
					\node[circle, draw, fill=black, inner sep=0pt, minimum width=1.5pt] (f) at (2, -1) {};
					
					\draw (d) to (e);
					\draw[densely dashed] (e) to [bend left=-80] (f);

					\node[circle, draw, fill=black, inner sep=0pt, minimum width=1.5pt] (n) at (4.25, -0.7) {};
					\node[circle, draw, fill=black, inner sep=0pt, minimum width=1.5pt] (o) at (4, -0.7) {};
					\node[circle, draw, fill=black, inner sep=0pt, minimum width=1.5pt] (p) at (4.5, -0.7) {};
					
					\draw 
					(n) to (p)
					(n) to (o);
					\draw[densely dashed] (o) to [bend left=-80] (p);

					\node[circle, draw, fill=black, inner sep=0pt, minimum width=1.5pt] (q) at (-.75, -1.15) {};
					\node[circle, draw, fill=black, inner sep=0pt, minimum width=1.5pt] (r) at (-.5, -1.15) {};
					\node[circle, draw, fill=black, inner sep=0pt, minimum width=1.5pt] (s) at (-1, -1.15) {};
					
					\draw(r) to [bend left=80] (s);

					\node[circle, draw, fill=black, inner sep=0pt, minimum width=1.5pt] (t) at (-.75, 1.15) {};
					\node[circle, draw, fill=black, inner sep=0pt, minimum width=1.5pt] (u) at (-.5, 1.15) {};
					\node[circle, draw, fill=black, inner sep=0pt, minimum width=1.5pt] (v) at (-1, 1.15) {};
				\end{tikzpicture}
				
				\columnbreak
				\null \vfill
				\tikzstyle{every node}=[circle, draw, fill=black, inner sep=0pt, minimum width=4pt]
				\begin{tikzpicture}
					\node[label={[shift={(0,-0.4)}]\scriptsize{1}}] (1) {};
					\node[label={[shift={(0,-0.4)}]\scriptsize{2}}] (2) [right of=1] {};
					\node[label={[shift={(0,-0.4)}]\scriptsize{3}}] (3) [right of=2] {};
					\node[label={[shift={(0,-0.4)}]\scriptsize{4}}] (4) [right of=3] {};
					\node[label={[shift={(0,-0.4)}]\scriptsize{5}}] (5) [right of=4] {};
					
					\draw[]
					(3) to (4)
					(4) to (5);
					\draw[dashed, bend right=55]
					(4) to (2)
					(5) to (3);
					\draw[dashed, bend right=55]
					(1) to (4);
					\draw[bend right=55]
					(2) to (5);
				\end{tikzpicture}
				\vfill \null
			\end{multicols}
			\caption{Partially compressed arc digraph (left) and corresponding gupcycle (right) for labeled graphs on $[3]$.}
			\label{fig:my_label}
		\end{figure}
	\end{example}
	
	\begin{theorem}\label{thm:tourcollection}
		For $n \ge 3$, if $t$ is the total length of a collection of edge-disjoint tours in $D_n^*$ (or, equivalently, any balanced subgraph of $D_n^*$), then we can construct a gupcycle for labeled graphs on $[n]$ of length $2^{\binom{n}{2}} - t$ by compression of twins.
	\end{theorem} 
	
	\begin{proof}
		Suppose $\mathcal{T}$ is a collection of edge-disjoint tours in the fully compressed arc digraph $D_n^*$ (or, equivalently, any balanced subgraph of $D_n^*$). Let $\mathcal{E}(\mathcal{T})$ be the set of edges in $\mathcal{T}$. Let $t = \abs{\mathcal{E}(\mathcal{T})}$. We construct the desired gupcycle as follows.
		
		In $D_n$, compress the pairs of twins that are edges in $\mathcal{E}(\mathcal{T})$. Call the resulting digraph $D_{\mathcal{T}}$. Then $D_{\mathcal{T}}$ is strongly connected because compressing twins cannot reduce connectivity, and $D_n$ is strongly connected. The digraph $D_{\mathcal{T}}$ is balanced because $D_n$ and $\mathcal{T}$ are balanced. Therefore, $D_{\mathcal{T}}$ is Eulerian.
		
		By Lemmas \ref{lem:gluing} and \ref{lem:cyclicgluing}, an Euler tour of $D_{\mathcal{T}}$ of length greater than $2(n-1)$ corresponds to a gupcycle of labeled graphs on $[n]$ where the edges of $D_{\mathcal{T}}$ that were compressed are labeled graphs on $[n]$ that have diamond edges at distance $n-1$ to represent both twins in the same $n$-window. This gupcycle has length equal to the number of edges in $D_{\mathcal{T}}$, which is $2^{\binom{n}{2}} - t$. The length condition of Lemma \ref{lem:cyclicgluing}, here equivalent to $t < 2^{\binom{n}{2}} - 2(n-1)$, is satisfied for all $n \ge 4$ because $t \le e(D_n^*) = 2^{\binom{n}{2}-1}$. For $n = 3$ it is satisfied for all $t \le 3$. The only remaining case is $n=3$ and $t=4$, and this cyclic gluing can also be done: Figure \ref{fig:smallestgupcycle} shows the gupcycle representing the unique Euler tour of $D_3^*$.
	\end{proof}
	
	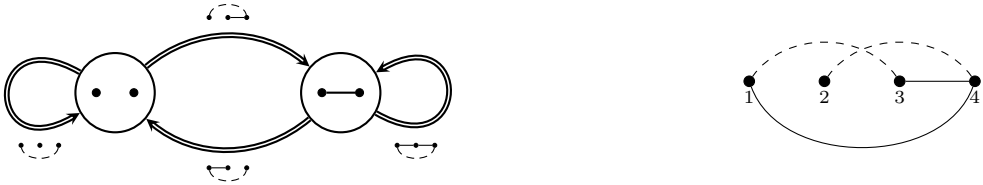
\begin{figure}[h!]
		\centering
		\begin{multicols}{2}
			\tikzstyle{every node}=[circle, draw, fill=black, inner sep=0pt, minimum width=3pt]
			\begin{tikzpicture}[]
				
				\node[circle, draw, thick, fill=white, inner sep=0pt, minimum width=30pt] (0) at (.25,0) {};
				\node[circle, draw, thick, fill=white, inner sep=0pt, minimum width=30pt] (00) at (3.25,0) {};

				\draw[->, -stealth, thick, double] (0) to [bend right=-40] (00);
				\draw[->, -stealth, thick, double, bend left=40]
				(00) to (0);

				\draw[thick, ->, -stealth,double] (0) to [out=150,in=210,looseness=7] (0);
				\draw[thick, ->, -stealth,double] (00) to [out=-30,in=30,looseness=7] (00);
				
				\node (1) at (0,0) {};
				\node (2) at (.5,0) {};
				\node (3) at (3, 0) {};
				\node (4) at (3.5,0) {};
				
				\draw[thick]
				(3) to (4);

				\node[circle, draw, fill=black, inner sep=0pt, minimum width=1.5pt] (a) at (1.75, 1) {};
				\node[circle, draw, fill=black, inner sep=0pt, minimum width=1.5pt] (b) at (1.5, 1) {};
				\node[circle, draw, fill=black, inner sep=0pt, minimum width=1.5pt] (c) at (2, 1) {};

				\draw (a) to (c);
				\draw[densely dashed] (b) to [bend left=80] (c);

				\node[circle, draw, fill=black, inner sep=0pt, minimum width=1.5pt] (d) at (1.75, -1) {};
				\node[circle, draw, fill=black, inner sep=0pt, minimum width=1.5pt] (e) at (1.5, -1) {};
				\node[circle, draw, fill=black, inner sep=0pt, minimum width=1.5pt] (f) at (2, -1) {};
				
				\draw (d) to (e);
				\draw[densely dashed] (e) to [bend left=-80] (f);

				\node[circle, draw, fill=black, inner sep=0pt, minimum width=1.5pt] (n) at (4.25, -0.7) {};
				\node[circle, draw, fill=black, inner sep=0pt, minimum width=1.5pt] (o) at (4, -0.7) {};
				\node[circle, draw, fill=black, inner sep=0pt, minimum width=1.5pt] (p) at (4.5, -0.7) {};
				
				\draw 
				(n) to (p)
				(n) to (o);
				\draw[densely dashed] (o) to [bend left=-80] (p);

				\node[circle, draw, fill=black, inner sep=0pt, minimum width=1.5pt] (q) at (-.75, -0.7) {};
				\node[circle, draw, fill=black, inner sep=0pt, minimum width=1.5pt] (r) at (-.5, -0.7) {};
				\node[circle, draw, fill=black, inner sep=0pt, minimum width=1.5pt] (s) at (-1, -0.7) {};
				
				\draw[densely dashed](r) to [bend left=80] (s);
			\end{tikzpicture}
			
			\columnbreak
			\null \vfill
			\tikzstyle{every node}=[circle, draw, fill=black, inner sep=0pt, minimum width=4pt]
			\begin{tikzpicture}
				\node[label={[shift={(0,-0.4)}]\scriptsize{1}}] (1) {};
				\node[label={[shift={(0,-0.4)}]\scriptsize{2}}] (2) [right of=1] {};
				\node[label={[shift={(0,-0.4)}]\scriptsize{3}}] (3) [right of=2] {};
				\node[label={[shift={(0,-0.4)}]\scriptsize{4}}] (4) [right of=3] {};
				
				\draw[]
				(3) to (4);
				\draw[dashed, bend right=55]
				(3) to (1)
				(4) to (2);
				\draw[bend right=70]
				(1) to (4);
			\end{tikzpicture}
			\vfill \null
		\end{multicols}
		\caption{Fully compressed arc digraph (left) and corresponding gupcycle (right) for labeled graphs on $[3]$.}
		\label{fig:smallestgupcycle}
	\end{figure}
	
	For every $n$, we have a gupcycle of length $2^{\binom{n}{2}-1}$, corresponding to an Euler tour of $D_n^*$, and we have a gucycle of length $2^{\binom{n}{2}}$, corresponding to an Euler tour of $D_n$. Theorem \ref{thm:tourcollection} enables us to find gupcycles of many different lengths in between these extremes. It would be interesting to know if all lengths between $2^{\binom{n}{2}-1}$ and $2^{\binom{n}{2}}$ are possible. To determine possible gupcycle lengths, we will count the loops and the tours of length $2$ in $D_n^*$, as collections of them are always edge-disjoint.

	\begin{lemma}\label{lem:loopcount}
		For $n \ge 2$, the fully compressed arc digraph $D_n^*$ for labeled graphs on $[n]$ contains exactly $2^{n-2}$ loops.
	\end{lemma}
	
	\begin{proof}
		By definition $L$ is a loop in $D_n^*$ if and only if $L[1,2,\ldots,n-1] \oiso L[2,3,\ldots,n]$ and $L$ has diamond edge set $\set{\set{1,n}}$. The neighborhood of $1$ determines the neighborhood of $2$, which determines the neighborhood of $3$, and so on. By induction $L$ is determined by the neighborhood of $1$. Vertex $1$ in $L$ may be adjacent to any of the $n-2$ vertices in $\set{2,3,\ldots,n-1}$, so it has $2^{n-2}$ possible neighborhoods. 
		Thus, there are $2^{n-2}$ loops in the fully compressed arc digraph $D_n^*$.
	\end{proof}
	
	\begin{lemma}\label{lem:2tourcount} 
		For $n \ge 2$, the fully compressed arc digraph $D_n^*$ for labeled graphs on $[n]$ contains exactly $2^{2n-5} - 2^{n-3}$ tours of length $2$.
	\end{lemma}
	
	\begin{proof}For $n=2$ there are no tours of length $2$ since there is only one vertex in $D_2^*$. Suppose $n \ge 3$. The tours of length $2$ in $D_n^*$ are exactly the unordered pairs $\set{H,K} \in \binom{V(D_n^*)}{2}$ such that 
		\begin{align*}
			K[1,2,\ldots,n-2] &\oiso H[2,3,\ldots,n-1]\quad\text{and}\\
			H[1,2,\ldots,n-2] &\oiso K[2,3,\ldots,n-1].
		\end{align*}
		
		To count these unordered pairs of distinct labeled graphs, we first count the ordered pairs $(H,K)$ such that $(H,K)$ and $(K,H)$ both are edges of $D_n^*$, and we will include the case $H=K$. Since $(H,K)$ overlaps, there is a unique labeled partial graph $J$ on $[n]$ which is the gluing of $(H,K)$ plus the diamond edge $\set{1,n}$. The labeled partial graph $J$ is the edge of $D_n^*$ from $H$ to $K$ and satisfies
		\begin{align*}
			J[1, 2, \dots, n-1] &\oiso H\quad\text{and}\\
			J[2, 3, \dots, n] &\oiso K.
		\end{align*}
		Since $D_n^*$ also has an edge from $K$ to $H$, the pair $(K,H)$ overlaps, so $J$ must also satisfy
		\[
		J[1,2,\ldots,n-2] \oiso H[1,2,\ldots,n-2] \oiso K[2,3,\ldots,n-1] \oiso J[3,4,\ldots,n]. \\
		\]
		In other words, we are counting the labeled partial graphs $J$ on $[n]$ with $E_{\diam}(J) = \set{\set{1,n}}$ such that $(J,J)$ overlaps by $n-2$.
		
		We will show that there are $2^{2n-4}$ such partial graphs $J$. Vertex $1$ in $J$ may be adjacent to any of the $n-2$ vertices in $\set{2,3,\ldots,n-1}$ (recall $\set{1,n}$ must be a diamond edge), so it has $2^{n-2}$ possible neighborhoods. Given any such neighborhood for vertex $1$, there are $2^{n-2}$ possible neighborhoods for vertex $2$, corresponding to its neighbors among the vertices in  $\set{3,4,\ldots,n}$. Since $J[1,2,\ldots,n-2] \oiso J[3,4,\ldots,n]$, the neighborhoods of vertices $1$ and $2$ determine the neighborhoods of all the odd and even vertices of $J$, respectively. Therefore there are $2^{n-2}\cdot 2^{n-2} = 2^{2n-4}$ such partial graphs $J$, each representing an ordered pair $(H,K)$ such that there is a walk of length $2$ between vertices $H$ and $K$.
		
		In this count of $2^{2n-4}$ ordered pairs $(H,K)$, we included pairs in which $H=K$, of which there are $2^{n-2}$ by Lemma \ref{lem:loopcount}. Thus there are $2^{2n-4}-2^{n-2}$ ordered pairs $(H,K)$ with $H\ne K$. Each such pair was counted twice, once for each ordering, so there are $2^{2n-5}-2^{n-3}$ unordered pairs $\set{H,K}$.
	\end{proof}

	\begin{theorem}\label{thm:guplengths}
		For every $n \ge 3$, for every $\ell$ in $2^{\binom{n}{2}-1} \le \ell \le 2^{\binom{n}{2}-1} + 2^{2n-4}$ or $2^{\binom{n}{2}} - 2^{2n-4}\le \ell \le 2^{\binom{n}{2}}$, we can construct a gupcycle of length $\ell$ for labeled graphs on $[n]$ by compression of twins.
	\end{theorem}
	
	\begin{proof}Let $t = 2^{\binom{n}{2}}-\ell$. We will show how to construct a gupcycle of length $\ell = 2^{\binom{n}{2}}-t$, reducing the length of a gucycle by $t$, for every $t$ in $0 \le t \le 2^{2n-4}$ or $2^{\binom{n}{2}-1} - 2^{2n-4}\le t \le 2^{\binom{n}{2}-1}$.
		
		Suppose $0 \le t \le 2^{2n-4}$. If $t \le 2^{n-2}$, then by Lemma \ref{lem:loopcount} let $\mathcal{T}$ be a collection of $t$ loops in $D_n^*$. Otherwise, $2^{n-2} +1 \le t \le 2^{2n-4}$. If $t$ is even, let $\mathcal{T}$ be a collection of $\frac{t}{2} - 2^{n-3}$ tours of length $2$ (by Lemma \ref{lem:2tourcount}) and all $2^{n-2}$ loops in $D_n^*$, so $\mathcal{T}$ has total length $2(\frac{t}{2} - 2^{n-3})+2^{n-2}=t$. If $t$ is odd, then let $\mathcal{T}$ be a collection of $\frac{t+1}{2} - 2^{n-3}$ tours of length $2$ and $2^{n-2}-1$ loops, so $\mathcal{T}$ has total length $2(\frac{t+1}{2} - 2^{n-3})+2^{n-2}-1=t$. We have shown that for all $0 \le t \le 2^{2n-4}$, there is a collection of edge-disjoint tours in $D_n^*$ of total length $t$, so Theorem \ref{thm:tourcollection} guarantees the existence of a gupcycle of length $2^{\binom{n}{2}}-t$.
		
		If $2^{\binom{n}{2}-1} - 2^{2n-4}\le t \le 2^{\binom{n}{2}-1}$, then let $s = 2^{\binom{n}{2}-1}-t$, so $0 \le s \le 2^{2n-4}$. Let $\mathcal{S}$ be a collection of edge-disjoint tours in $D_n^*$ of total length $s$, guaranteed by the previous paragraph. Let $\mathcal{T} = E(D_n^*)\setminus \mathcal{E}(\mathcal{S})$, the set of edges of $D_n^*$ that are not in $\mathcal{S}$. Since $D_n^*$ and $\mathcal{S}$ are balanced, $\mathcal{T}$ represents a balanced subgraph of $D_n^*$, and it contains $t$ edges. By Theorem \ref{thm:tourcollection}, there is a gupcycle of length $2^{\binom{n}{2}}-t$.
	\end{proof}
	
	For $n=3$ and $n=4$, Theorem \ref{thm:guplengths} provides gupcycles of all possible lengths using compression of twins.
	
	\begin{corollary}
		For every $4 \le \ell \le 8$, we can construct a gupcycle of length $\ell$ for labeled graphs on $[3]$. For every $32 \le \ell \le 64$, we can construct a gupcycle of length $\ell$ for labeled graphs on $[4]$.
	\end{corollary}
	
	\begin{question}\label{qu:labeled}
		For $n \ge 5$, do there exist gupcycles for labeled graphs on $[n]$ of all lengths $\ell$ such that $2^{\binom{n}{2}-1}+2^{2n-4} < \ell < 2^{\binom{n}{2}}-2^{2n-4}$?
	\end{question}

	\section{Unlabeled Graphs}\label{sec:unlabeled}

	\subsection{Gucycles for Unlabeled Graphs}
	Brockman, Kay, and Snively stated the following conjecture.
	\begin{conjecture}[Brockman, Kay, Snively\cite{BKS10}]\label{conj:graphs}
		For each $n \neq 2$, there exists a gucycle of isomorphism classes of graphs on $n$ vertices.
	\end{conjecture}
	In support of this conjecture, they found a guycle for $n=4$. The conjecture as stated is false for $n=3$. There are $4$ graphs on $4$ vertices up to isomorphism. Hence any gucycle must have $4$ vertices, but it is not possible to realize both $K_3$ and $\overline{K_3}$ as subgraphs on $4$ vertices. However, it is possible to create a guword instead of a gucycle. See Figure~\ref{fig:gupword}.
	
	\begin{figure}[h!]
		\centering
		\tikzstyle{every node}=[circle, draw, fill=black, inner sep=0pt, minimum width=4pt]
		
		\begin{tikzpicture}
			
			\node[label={[shift={(0,-0.4)}]\scriptsize{1}}] (1) {};
			\node[label={[shift={(0,-0.4)}]\scriptsize{2}}] (2) [right of=1] {};
			\node[label={[shift={(0,-0.4)}]\scriptsize{3}}] (3) [right of=2] {};
			\node[label={[shift={(0,-0.4)}]\scriptsize{4}}] (4) [right of=3] {};
			\node[label={[shift={(0,-0.4)}]\scriptsize{5}}] (5) [right of=4] {};
			\node[label={[shift={(0,-0.4)}]\scriptsize{6}}] (6) [right of=5] {};
			
			\draw[]
			(3) to (4)
			(4) to (5)
			(5) to (6);
			\draw[bend left=55]
			(3) to (5);
		\end{tikzpicture}    
		
		\caption{A guword for isomorphism classes of graphs on $3$ vertices.}
		\label{fig:gupword}
	\end{figure}
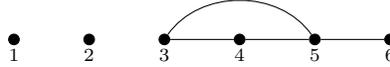
	
	The corrected version of Conjecture~\ref{conj:graphs} is as follows.
	\begin{conjecture}\label{conj:unlabeled4}
		For each $n \geq 4$, there exists a gucycle of isomorphism classes of graphs on $n$ vertices.
	\end{conjecture}
	In support of this conjecture, we found a gucycle for unlabeled graphs on $5$ vertices (Figure \ref{fig:5gucycle}) using a computer program written by Bernard Lidick\'{y}.
	Finding a gucycle for unlabeled graphs on 6 vertices was beyond what the program could do in a short time. 
	
	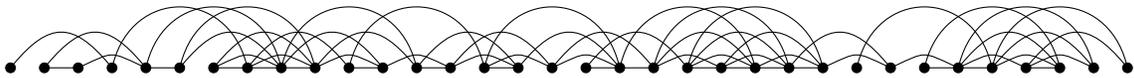
\begin{figure}[h!]
		\centering
		
		\tikzset{
			vtx/.style={inner sep=1.3pt, outer sep=0pt, circle, fill=black,draw=black}, 
		}
		\newcommand{\ledge}[2]{
			\pgfmathtruncatemacro{\angle}{-25+(#2-#1)*25}
			\draw (v#1) to[bend left=\angle,looseness=1.5] (v#2);
		}
		
		\begin{tikzpicture}[scale=0.45]
			\foreach \x in {0,...,33}{
				\draw (\x,0) node[vtx](v\x){};
			}
			\ledge{0}{3} \ledge{1}{2} \ledge{1}{4} \ledge{2}{4} \ledge{3}{7} \ledge{4}{5} \ledge{4}{8} \ledge{5}{8} \ledge{5}{9} \ledge{6}{7} \ledge{6}{8} \ledge{6}{9} \ledge{7}{8} \ledge{7}{9} \ledge{7}{10} \ledge{8}{9} \ledge{8}{11} \ledge{8}{12} \ledge{9}{11} \ledge{9}{12} \ledge{10}{11} \ledge{10}{14} \ledge{11}{13} \ledge{12}{13} \ledge{12}{15} \ledge{13}{15} \ledge{13}{16} \ledge{14}{15} \ledge{14}{16} \ledge{14}{18} \ledge{15}{18} \ledge{16}{18} \ledge{16}{19} \ledge{17}{18} \ledge{17}{20} \ledge{18}{19} \ledge{18}{22} \ledge{19}{21} \ledge{19}{22} \ledge{19}{23} \ledge{20}{21} \ledge{20}{22} \ledge{20}{23} \ledge{20}{24} \ledge{21}{22} \ledge{21}{23} \ledge{21}{24} \ledge{22}{23} \ledge{22}{24} \ledge{23}{24} \ledge{23}{26} \ledge{24}{26} \ledge{25}{29} \ledge{26}{28} \ledge{27}{28} \ledge{27}{31} \ledge{28}{29} \ledge{28}{30} \ledge{28}{31} \ledge{28}{32} \ledge{29}{31} \ledge{29}{32} \ledge{29}{33} \ledge{30}{31} \ledge{30}{32} \ledge{30}{33}
			
		\end{tikzpicture}    
		
		\caption{A graph universal cycle for unlabeled graphs on $5$ vertices.}
		\label{fig:5gucycle}
	\end{figure}
	
	\subsection{$f$-fold Gucycles for Unlabeled Graphs}\label{sec:ffold}
	
	In analogy to $f$-fold De Bruijn cycles, introduced by Tesler \cite{Tesler2017}, we define $f$-fold gucycles as follows. 
	
	\begin{definition}
		An \emph{$f$-fold gucycle} for a family of $n$-vertex graphs $\mathcal{F}$ is a cyclically ordered graph $G$ that contains each graph in $\mathcal{F}$ exactly $f$ times.
	\end{definition}
	
	Gucycles, therefore, are the same as $1$-fold gucycles. By generalizing from gucycles to $f$-fold gucycles, we prove a weaker version of Conjecture \ref{conj:graphs}.
	
	\begin{theorem}
		For every $n\ge 3$, there is an $f$-fold gucycle for $n$-vertex unlabeled graphs for some $f \ge 1$.
	\end{theorem}
	
	\begin{proof}
		Let $D=D_n$ be the arc digraph for labeled graphs on $[n]$. Let $G$ be an $n$-vertex unlabeled graph. Let $D[G]$ be the subgraph of $D$ induced by the edges that are labeled graphs on $[n]$ in the isomorphism class $G$. We will first show that $D[G]$ is balanced. Suppose $H$ is an $(n-1)$-vertex labeled graph that is a vertex of $D[G]$. Let $J$ be any edge of $D[G]$ whose head is $H$; in other words, suppose $J$ is a labeled graph on $[n]$ isomorphic to $G$ such that $J[2, \ldots, n] \oiso H$. We define $\phi(J)$ to be the labeled graph on $[n]$ obtained by subtracting $1$ from each vertex label and then relabeling $0$ as $n$. By definition $\phi$ is a graph isomorphism, so $\phi(J)$ is also isomorphic to $G$. Moreover, writing $J' = \phi(J)$, we have $J'[1, \ldots, n-1] \oiso H$, so $J'$ is an edge of $D[G]$ whose tail is $H$. The inverse of $\phi$ similarly adds $1$ to each vertex label and then relabels $n+1$ as $1$. Since $\phi$ is a bijection from the edges of $D[G]$ pointing to an arbitrary vertex $H$ to the edges of $D[G]$ pointing from $H$, we conclude that $D[G]$ is balanced. Loops of $D[G]$ are fixed points of $\phi$ and contribute to both the in-degree and the out-degree of their vertices.
		
		For each isomorphism class $G$, let $f_G$ be the number of edges in $D[G]$, or equivalently the number of distinct labeled graphs isomorphic to $G$. Let $f$ be the least common multiple of the set of $f_G$'s for all $n$-vertex unlabeled graphs $G$.
		
		For each edge $L$ of $D$, let $G_L$ be the isomorphism class of $L$, and add $f/f_{G_L}-1$ copies of the edge $L$ to $D$. Let $M$ be the resulting directed multigraph. We have added only copies of existing edges, so $M$ retains the property that if the head of $u$ is the tail of $v$ then $u[2,\ldots,n]\oiso v[1,\ldots,n-1]$. The directed multigraph $M$ is connected because it contains the connected digraph $D$. As $D[G]$ is balanced for every isomorphism class $G$, we added the same number of in-edges and out-edges to each vertex $H$ of $D$ to produce $M$. (Multiple isomorphism classes $G$ may contribute to this number for a given vertex $H$.) Since $D$ is balanced, $M$ is balanced and Eulerian. In addition, $M$ contains exactly $(1 + f/f_{G_L}-1)f_{G_L} = f$ edges representing each $n$-vertex unlabeled graph $G_L$. An Euler tour in $M$ traverses each edge $L$ exactly once, so traverses each isomorphism class $G_L$ exactly $f$ times, and therefore is equivalent to an $f$-fold gucycle for $n$-vertex unlabeled graphs.
	\end{proof}
	
	\begin{example}We will show how to find an $f$-fold gucycle for $3$-vertex unlabeled graphs using the ideas of the above proof. In Figure \ref{fig:3foldarc}, we show the arc digraph $D_3$ and the directed multigraph $M$. Here $f_{P_3} = 3$, $f_{K_2\cup K_1} = 3$, $f_{K_3} = 1$, and $f_{\overline{K_3}} = 1$, so $f = 3$.
		
		\begin{figure}[h!]
			\centering
			\begin{multicols}{2}
				\tikzstyle{every node}=[circle, draw, fill=black, inner sep=0pt, minimum width=3pt]
				\begin{tikzpicture}[]
					\node[circle, draw, thick, fill=white, inner sep=0pt, minimum width=30pt] (0) at (.25,0) {};
					\node[circle, draw, thick, fill=white, inner sep=0pt, minimum width=30pt] (00) at (3.25,0) {};
					
					\draw[->, -stealth, thick] (00) to [bend right=-40] (0);
					\draw[->, -stealth, thick, bend left=60]
					(00) to (0);
					\draw[->, -stealth, thick, bend left=40]
					(0) to (00);
					\draw[->, -stealth, thick, bend left=60]
					(0) to (00);
					\draw[thick, ->, -stealth] (0) to [out=115,in=165,looseness=7] (0);
					\draw[thick, ->, -stealth] (0) to [out=245,in=195,looseness=7] (0);
					\draw[thick, ->, -stealth] (00) to [out=65,in=15,looseness=7] (00);
					\draw[thick, ->, -stealth] (00) to [out=-65,in=-15,looseness=7] (00);
					
					\node (1) at (0,0) {};
					\node (2) at (.5,0) {};
					\node (3) at (3, 0) {};
					\node (4) at (3.5,0) {};
					\draw[thick]
					(3) to (4);
					
					\node[circle, draw, fill=black, inner sep=0pt, minimum width=1.5pt] (a) at (1.75, 1.25) {};
					\node[circle, draw, fill=black, inner sep=0pt, minimum width=1.5pt] (b) at (1.5, 1.25) {};
					\node[circle, draw, fill=black, inner sep=0pt, minimum width=1.5pt] (c) at (2, 1.25) {};
					\draw (a) to (c);
					\draw(b) to [bend left=60] (c);
					
					\node[circle, draw, fill=black, inner sep=0pt, minimum width=1.5pt] (k) at (1.75, .5) {};
					\node[circle, draw, fill=black, inner sep=0pt, minimum width=1.5pt] (l) at (1.5, .5) {};
					\node[circle, draw, fill=black, inner sep=0pt, minimum width=1.5pt] (m) at (2, .5) {};
					\draw (k) to (m); 
					
					\node[circle, draw, fill=black, inner sep=0pt, minimum width=1.5pt] (d) at (1.75, -1.25) {};
					\node[circle, draw, fill=black, inner sep=0pt, minimum width=1.5pt] (e) at (1.5, -1.25) {};
					\node[circle, draw, fill=black, inner sep=0pt, minimum width=1.5pt] (f) at (2, -1.25) {};
					\draw (d) to (e);
					\draw(e) to [bend left=-60] (f);
					
					\node[circle, draw, fill=black, inner sep=0pt, minimum width=1.5pt] (h) at (1.75, -.5) {};
					\node[circle, draw, fill=black, inner sep=0pt, minimum width=1.5pt] (i) at (1.5, -.5) {};
					\node[circle, draw, fill=black, inner sep=0pt, minimum width=1.5pt] (j) at (2, -.5) {};
					\draw (h) to (i);
					
					\node[circle, draw, fill=black, inner sep=0pt, minimum width=1.5pt] (n) at (4.25, -1.15) {};
					\node[circle, draw, fill=black, inner sep=0pt, minimum width=1.5pt] (o) at (4, -1.15) {};
					\node[circle, draw, fill=black, inner sep=0pt, minimum width=1.5pt] (p) at (4.5, -1.15) {};
					\draw 
					(n) to (p)
					(n) to (o);
					\draw(o) to [bend left=-60] (p);
					
					\node[circle, draw, fill=black, inner sep=0pt, minimum width=1.5pt] (q) at (-.75, -1.15) {};
					\node[circle, draw, fill=black, inner sep=0pt, minimum width=1.5pt] (r) at (-.5, -1.15) {};
					\node[circle, draw, fill=black, inner sep=0pt, minimum width=1.5pt] (s) at (-1, -1.15) {};
					\draw(r) to [bend left=60] (s);
					
					\node[circle, draw, fill=black, inner sep=0pt, minimum width=1.5pt] (t) at (-.75, 1.15) {};
					\node[circle, draw, fill=black, inner sep=0pt, minimum width=1.5pt] (u) at (-.5, 1.15) {};
					\node[circle, draw, fill=black, inner sep=0pt, minimum width=1.5pt] (v) at (-1, 1.15) {};
					
					\node[circle, draw, fill=black, inner sep=0pt, minimum width=1.5pt] (w) at (4.25, 1.15) {};
					\node[circle, draw, fill=black, inner sep=0pt, minimum width=1.5pt] (x) at (4, 1.15) {};
					\node[circle, draw, fill=black, inner sep=0pt, minimum width=1.5pt] (y) at (4.5, 1.15) {};
					\draw (w) to (x)
					(w) to (y);
				\end{tikzpicture}
				
				\columnbreak
				
				\begin{tikzpicture}[]
					\node[circle, draw, thick, fill=white, inner sep=0pt, minimum width=30pt] (0) at (.25,0) {};
					\node[circle, draw, thick, fill=white, inner sep=0pt, minimum width=30pt] (00) at (3.25,0) {};
					
					\draw[->, -stealth, thick] (00) to [bend right=-40] (0);
					\draw[->, -stealth, thick, bend left=60]
					(00) to (0);
					\draw[->, -stealth, thick, bend left=40]
					(0) to (00);
					\draw[->, -stealth, thick, bend left=60]
					(0) to (00);
					\draw[thick, ->, -stealth, double distance=2pt] (0) to [out=115,in=165,looseness=7] (0);
					\draw[draw=white, double=black, thick, arrows={-Stealth[length=2mm, width=2mm, black]}] (0) to [out=115,in=165,looseness=7] (0);
					\draw[thick, ->, -stealth] (0) to [out=245,in=195,looseness=7] (0);
					\draw[thick, ->, -stealth] (00) to [out=65,in=15,looseness=7] (00);
					\draw[thick, ->, -stealth, double distance=2pt] (00) to [out=-65,in=-15,looseness=7] (00);
					\draw[draw=white, double=black, thick, arrows={-Stealth[length=2mm, width=2mm, black]}] (00) to [out=-65,in=-15,looseness=7] (00);
					
					\node (1) at (0,0) {};
					\node (2) at (.5,0) {};
					\node (3) at (3, 0) {};
					\node (4) at (3.5,0) {};
					\draw[thick]
					(3) to (4);
					
					\draw[red, loosely dashed] (1.4, 1.15) rectangle (2.1, 1.45);
					\node[circle, draw, fill=black, inner sep=0pt, minimum width=1.5pt] (a) at (1.75, 1.25) {};
					\node[circle, draw, fill=black, inner sep=0pt, minimum width=1.5pt] (b) at (1.5, 1.25) {};
					\node[circle, draw, fill=black, inner sep=0pt, minimum width=1.5pt] (c) at (2, 1.25) {};
					\draw (a) to (c);
					\draw(b) to [bend left=60] (c);
					
					\draw[blue, densely dotted] (1.4, .4) rectangle (2.1, .6);
					\node[circle, draw, fill=black, inner sep=0pt, minimum width=1.5pt] (k) at (1.75, .5) {};
					\node[circle, draw, fill=black, inner sep=0pt, minimum width=1.5pt] (l) at (1.5, .5) {};
					\node[circle, draw, fill=black, inner sep=0pt, minimum width=1.5pt] (m) at (2, .5) {};
					\draw (k) to (m); 
					
					\draw[red, loosely dashed] (1.4, -1.15) rectangle (2.1, -1.45);
					\node[circle, draw, fill=black, inner sep=0pt, minimum width=1.5pt] (d) at (1.75, -1.25) {};
					\node[circle, draw, fill=black, inner sep=0pt, minimum width=1.5pt] (e) at (1.5, -1.25) {};
					\node[circle, draw, fill=black, inner sep=0pt, minimum width=1.5pt] (f) at (2, -1.25) {};
					\draw (d) to (e);
					\draw(e) to [bend left=-60] (f);
					
					\draw[blue, densely dotted] (1.4, -.4) rectangle (2.1, -.6);
					\node[circle, draw, fill=black, inner sep=0pt, minimum width=1.5pt] (h) at (1.75, -.5) {};
					\node[circle, draw, fill=black, inner sep=0pt, minimum width=1.5pt] (i) at (1.5, -.5) {};
					\node[circle, draw, fill=black, inner sep=0pt, minimum width=1.5pt] (j) at (2, -.5) {};
					\draw (h) to (i);
					
					\draw[yellow] (3.9, -1.07) rectangle (4.6, -1.35);
					\node[circle, draw, fill=black, inner sep=0pt, minimum width=1.5pt] (n) at (4.25, -1.15) {};
					\node[circle, draw, fill=black, inner sep=0pt, minimum width=1.5pt] (o) at (4, -1.15) {};
					\node[circle, draw, fill=black, inner sep=0pt, minimum width=1.5pt] (p) at (4.5, -1.15) {};
					\draw 
					(n) to (p)
					(n) to (o);
					\draw(o) to [bend left=-60] (p);
					
					\draw[blue, densely dotted] (-.4, -1.05) rectangle (-1.1, -1.35);
					\node[circle, draw, fill=black, inner sep=0pt, minimum width=1.5pt] (q) at (-.75, -1.15) {};
					\node[circle, draw, fill=black, inner sep=0pt, minimum width=1.5pt] (r) at (-.5, -1.15) {};
					\node[circle, draw, fill=black, inner sep=0pt, minimum width=1.5pt] (s) at (-1, -1.15) {};
					\draw(r) to [bend left=60] (s);
					
					\draw[green, densely dashdotted] (-.4, 1.07) rectangle (-1.1, 1.25);
					\node[circle, draw, fill=black, inner sep=0pt, minimum width=1.5pt] (t) at (-.75, 1.16) {};
					\node[circle, draw, fill=black, inner sep=0pt, minimum width=1.5pt] (u) at (-.5, 1.16) {};
					\node[circle, draw, fill=black, inner sep=0pt, minimum width=1.5pt] (v) at (-1, 1.16) {};
					
					\draw[red, loosely dashed] (3.9, 1.05) rectangle (4.6, 1.25);
					\node[circle, draw, fill=black, inner sep=0pt, minimum width=1.5pt] (w) at (4.25, 1.15) {};
					\node[circle, draw, fill=black, inner sep=0pt, minimum width=1.5pt] (x) at (4, 1.15) {};
					\node[circle, draw, fill=black, inner sep=0pt, minimum width=1.5pt] (y) at (4.5, 1.15) {};
					\draw (w) to (x)
					(w) to (y);
				\end{tikzpicture}
			\end{multicols}
			\caption{Arc digraph $D_3$ (left) and directed multigraph $M$ (right). Isomorphism classes are shown in boxes of the same color. }
			\label{fig:3foldarc}
		\end{figure}
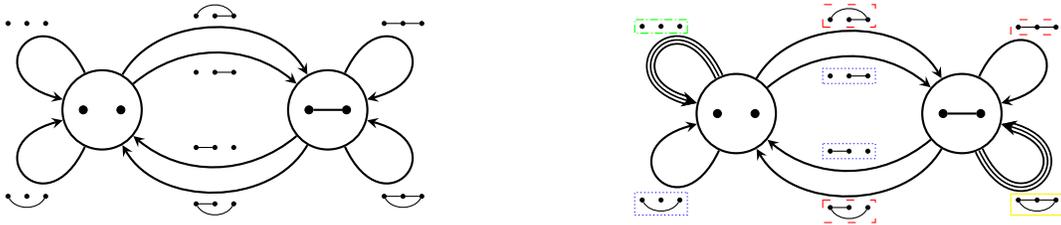

		Any Euler tour in the above directed multigraph $M$ is a $3$-fold gucycle for unlabeled graphs on $3$ vertices, an example of which is shown in Figure \ref{fig:3foldgucycle}.
		
		\begin{figure}[h!]
			\centering
			
			\tikzstyle{every node}=[circle, draw, fill=black, inner sep=0pt, minimum width=4pt]
			\begin{tikzpicture}[scale=0.4]
				\node[label={[shift={(0,-0.4)}]\scriptsize{1}}] (2) {};
				\node[label={[shift={(0,-0.4)}]\scriptsize{2}}] (3) [right of=2] {};
				\node[label={[shift={(0,-0.4)}]\scriptsize{3}}] (4) [right of=3] {};
				\node[label={[shift={(0,-0.4)}]\scriptsize{4}}] (5) [right of=4] {};
				\node[label={[shift={(0,-0.4)}]\scriptsize{5}}] (6) [right of=5] {};
				\node[label={[shift={(0,-0.4)}]\scriptsize{6}}] (7) [right of=6] {};
				\node[label={[shift={(0,-0.4)}]\scriptsize{7}}] (8) [right of=7] {};
				\node[label={[shift={(0,-0.4)}]\scriptsize{8}}] (9) [right of=8] {};
				\node[label={[shift={(0,-0.4)}]\scriptsize{9}}] (10) [right of=9] {};
				\node[label={[shift={(0,-0.4)}]\scriptsize{10}}] (11) [right of=10] {};
				\node[label={[shift={(0,-0.4)}]\scriptsize{11}}] (12) [right of=11] {};
				\node[label={[shift={(0,-0.4)}]\scriptsize{12}}] (1) [right of=12]{};
				\draw
				(4) to (5)
				(5) to (6)
				(6) to (7)
				(7) to (8)
				(10) to (11)
				(11) to (12);
				\draw[bend left=55]
				(3) to (5)
				(4) to (6)
				(5) to (7)
				(6) to (8)
				(7) to (9)
				(8) to (10);
			\end{tikzpicture}   
			
			\caption{A $3$-fold gucycle for $3$-vertex unlabeled graphs, constructed from Figure \ref{fig:3foldarc}.}
			\label{fig:3foldgucycle}
		\end{figure}
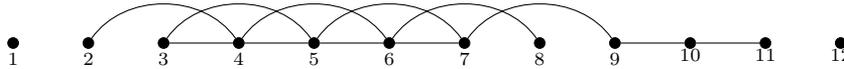
		
	\end{example}
	
	\subsection{Gupwords for Unlabeled Graphs}
	
	In Figure \ref{fig:gupwordunlabel} we present a gupword for $4$-vertex unlabeled graphs with one diamond edge. It is an open question whether gupcycles and gupwords exist for unlabeled $n$-vertex graphs in general.
	
	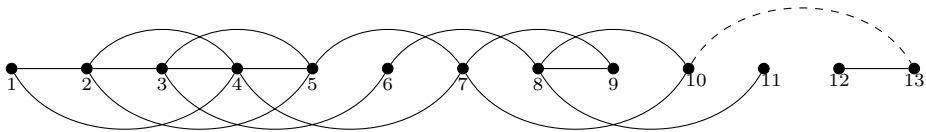
\begin{figure}[h!]
		\centering
		
		\tikzstyle{every node}=[circle, draw, fill=black, inner sep=0pt, minimum width=4pt]
		
		\begin{tikzpicture}
			
			\node[label={[shift={(0,-0.4)}]\scriptsize{1}}] (1) {};
			\node[label={[shift={(0,-0.4)}]\scriptsize{2}}] (2) [right of=1] {};
			\node[label={[shift={(0,-0.4)}]\scriptsize{3}}] (3) [right of=2] {};
			\node[label={[shift={(0,-0.4)}]\scriptsize{4}}] (4) [right of=3] {};
			\node[label={[shift={(0,-0.4)}]\scriptsize{5}}] (5) [right of=4] {};
			\node[label={[shift={(0,-0.4)}]\scriptsize{6}}] (6) [right of=5] {};
			\node[label={[shift={(0,-0.4)}]\scriptsize{7}}] (7) [right of=6] {};
			\node[label={[shift={(0,-0.4)}]\scriptsize{8}}] (8) [right of=7] {};
			\node[label={[shift={(0,-0.4)}]\scriptsize{9}}] (9) [right of=8] {};
			\node[label={[shift={(0.1,-0.4)}]\scriptsize{10}}] (10) [right of=9] {};
			\node[label={[shift={(0.1,-0.4)}]\scriptsize{11}}] (11) [right of=10] {};
			\node[label={[shift={(0,-0.4)}]\scriptsize{12}}] (12) [right of=11] {};
			\node[label={[shift={(0,-0.4)}]\scriptsize{13}}] (13) [right of=12] {};

			\path[every node/.style]
			(1) edge node {} (2)    
			(2) edge node {} (3)    
			(3) edge node {} (4)    
			(4) edge node {} (5)    
			(8) edge node {} (9)
			(12) edge node {} (13);
			
			\draw[bend left=-60]
			(1) to (4)
			(2) to (5)
			(3) to (6)
			(4) to (7)
			(7) to (10)
			(8) to (11);
			
			\draw[bend left=55]
			(2) to (4)
			(3) to (5)
			(5) to (7)
			(6) to (8)
			(7) to (9)
			(8) to (10);
			
			\draw [style=dashed, bend left=60]
			(10) to (13);
			
		\end{tikzpicture}
		
		\caption{Gupword for $4$-vertex unlabeled graphs.}
		\label{fig:gupwordunlabel}
		
	\end{figure}
	
	\section{Threshold Graphs}\label{sec:threshold}
	
	\begin{definition}
		A \emph{threshold graph} is a graph that can be constructed from a single vertex by repeatedly adding either an isolated vertex or a dominating vertex. We write $T_n$ for the set of all $n$-vertex threshold graphs.
	\end{definition}
	
	This definition implies that $T_n$ corresponds bijectively to $\set{0,1}^{n-1}$ as follows. Given any word in $\set{0,1}^{n-1}$, we can construct an $n$-vertex threshold graph. Start with a single vertex. Then, read the word from left to right, sequentially adding vertices to the graph, where the added vertex is isolated if the entry is $0$ and dominating if the entry is $1$. We call this binary word the \emph{binary representation} of the threshold graph. Given a threshold graph, the binary representation is uniquely determined by successively finding and deleting either a dominating vertex or an isolated vertex. See \cite{MP95, Keough15} for further discussion of characterizations of threshold graphs. The related concept of a $T_n$-universal threshold graph is studied in \cite{HK94}.
	
	\subsection{Gucycles for Threshold Graphs}
	
	Using binary representations of threshold graphs, we show that each De Bruijn cycle for $\set{0,1}^{n-1}$ produces a distinct gucycle for $n$-vertex threshold graphs.
	
	\begin{theorem}\label{thm:thresholdgucycle} For every $n \ge 4$, there exist at least $2^{2^{n-2}-(n-1)}$ gucycles for the $n$-vertex threshold graphs.
	\end{theorem}
	
	\begin{proof}
		We will describe an injective map from De Bruijn cycles for $\set{0,1}^{n-1}$ to gucycles for $n$-vertex threshold graphs. Let $w$ be a De Bruijn cycle for $\set{0,1}^{n-1}$; note $\abs{w} = 2^{n-1}$. Construct a graph $G$ on vertex set $[2^{n-1}]$ as follows. For each $i \in [2^{n-1}]$, if $w_i = 1$, then add the edges $ij$ for each $j \in \set{i-(n-1), i-(n-2), \ldots, i-1}$, where subtraction is done $\mod 2^{n-1}$. (If $w_i=0$, then do nothing.)
		
		We will show that $G$ is a gucycle. Let $T$ be an $n$-vertex threshold graph and $b \in \set{0,1}^{n-1}$ its binary representation. The condition $n \ge 4$ ensures that the number of vertices in $G$, $2^{n-1}$, is greater than $2(n-1)$, so if two vertices $x$ and $y$ appear in the same $n$-window of $G$ with $x$ before $y$, then there is no $n$-window of $G$ in which $y$ is before $x$. Thus, by the construction of $G$ and by the bijection between $T_n$ and $\set{0,1}^{n-1}$, the $n$-window of $G$ starting at vertex $i-1$ is the threshold graph $T$ if and only if the $(n-1)$-window of $w$ starting at position $i$ is $b$. Every word in $\set{0,1}^{n-1}$ appears exactly once as an $(n-1)$-window in $w$, so every $n$-vertex threshold graph appears exactly once as an $n$-window of $G$.
		
		To show this construction gives an injective map, we will describe the inverse. Let $G$ be a gucycle for $n$-vertex threshold graphs constructed in this way. By taking consecutive $n$-windows of $G$, we obtain a cyclic ordering of threshold graphs such that the consecutive pairs of threshold graphs overlap when ordered according to their binary representations. Taking the binary representation of each threshold graph in the cyclic ordering yields a unique cyclic ordering of the words in $\set{0,1}^{n-1}$ such that consecutive pairs of words of length $n-1$ overlap (i.e. the suffix of length $n-2$ of one word is the prefix of length $n-2$ of the next word). This ordering is a De Bruijn cycle for $\set{0,1}^{n-1}$. Thus no two distinct De Bruijn cycles for $\set{0,1}^{n-1}$ yield the same gucycle under our construction. By Theorem \ref{thm:DBcount}, there are $2^{2^{n-2}-(n-1)}$ De Bruijn cycles for $\set{0,1}^{n-1}$.
	\end{proof}
	
	Another way of viewing Theorem \ref{thm:thresholdgucycle} is that the subgraph of $D_n$ induced by the edges representing the $n$-vertex threshold graphs, labeled in the order of their binary encodings, is isomorphic to the De Bruijn graph for $\set{0,1}^{n-2}$ (whose edges are the binary words of length $n-1$). This subgraph of $D_n$ contains no twins, so we cannot create gupcycles through compression in the same way as for labeled graphs.
	
	\subsection{Gupcycles and Gupwords for Threshold Graphs}
	
	Extending the definitions from Section \ref{subsec:partial}, in this section, a \emph{diamond edge set} $\mathcal{D}$ for a graph $G = (V,E)$ is a family of disjoint subsets of $\binom{V}{2}\setminus E$. Each subset $D \in \mathcal{D}$ is thought of as being either all present or all absent. A \emph{partial graph} is an ordered triple $(V, E, \mathcal{D})$ where $G=(V,E)$ is a graph and $\mathcal{D}$ is a diamond edge set for $G$.
	We say graph $H$ is \emph{contained} in a (cyclically) ordered partial graph $G=(V, E, \mathcal{D})$ if $H$ is isomorphic as an ordered graph to a window of $G_{\mathcal{S}}$ for some $\emptyset \subseteq \mathcal{S} \subseteq \mathcal{D}$, where $G_{\mathcal{S}}$ is the (cyclically) ordered graph whose vertex set is $V$ and whose edge set is $E\cup (\cup_{S \in \mathcal{S}} S)$. A \emph{graph universal partial word (cycle)}, or \emph{gupword} \emph{(gupcycle)}, for a family of graphs $\mathcal{F}$ is a (cyclically) ordered partial graph $G$ that contains each graph in $\mathcal{F}$ exactly once. In the case where each of the elements of $\mathcal{D}$ is a singleton set consisting of one pair of vertices, these definitions are the same as Definitions \ref{def:partialgraph}, \ref{def:partialgraphcontain}, and \ref{def:gupcycle}.
	
	\begin{prop}\label{prop:gupword threshold} For every upword for $\set{0,1}^{n-1}$, there exists a gupword for $n$-vertex threshold graphs of the same length.
	\end{prop}
	
	\begin{proof}
		Let $w$ be an upword for $\set{0,1}^{n-1}$. Construct a partial graph $G$ on vertex set $V=\set{0,1,\dots,\abs{w}}$ as follows.  Let $I_i = \set{\set{i,j} \in \binom{V}{2} : \max\set{0,i-(n-1)} \le j \le i-1}$ for all $i\in[\abs{w}]$. Define the edge set as $E = \cup_{w_i = 1} I_i$ and the diamond edge set as $\mathcal{D} = \set{I_i : w_i = \diam}$. Comparing this definition of $G$ to the definition of threshold graph, we see that if an $n$-vertex graph $H$ is contained in $G$ then $H$ is a threshold graph.
		
		We will show $G$ is a gupword for $n$-vertex threshold graphs. Let $T$ be an $n$-vertex threshold graph and $b\in\set{0,1}^{n-1}$ its binary representation. We will show that the word $b$ is covered in $w$ starting at position $\ell$ if and only if the graph $T$ is contained in $G$ starting at position $\ell-1$.
		
		Let $\ell$ be the starting position of $b$ in $w$, and let $J_b = \set{I_i: w_i~=~\diam,\text{ } \ell \le i \le \ell+n-2,\text{ and } b_{i-\ell+1}=1} \subseteq \mathcal{D}$. Then the $n$-window starting at position $\ell-1$ of $G_{J_b}$, the graph on vertex set $V$ and edge set $E\cup\left(\bigcup_{I_i\in J_b} I_i\right)$, is exactly $T$ because it is constructed by successively adding dominating and isolated vertices according to the $0$'s and $1$'s of $b$.
		
		On the other hand, suppose $T$ is contained in an $n$-window of $G$. We define $\ell$ to be the index of the second position of the $n$-window, so that $T$ begins at index $\ell-1$. Let $J_T$ be the collection of sets $I_i$ such that $T$ is isomorphic as an ordered graph to this $n$-window of $G_{J_T}$. Recall the edges in the $n$-window of $G_{J_T}$ starting at $\ell-1$ are exactly the edges in $I_i$ for every $i \in \set{\ell-1,\ell,\ldots,\ell+n-2}$ such that either $w_i=1$, or both $w_i = \diam$ and $I_i \in J_T$. The binary representation $b$ of $T$ is covered in $w$ starting at position $\ell$ because $b$ has a $0$ exactly where $T$ has a (successive) isolated vertex and a $1$ exactly where $T$ has a (successive) dominating vertex. 
	\end{proof}
	
	Proposition \ref{prop:gupword threshold} is useful because binary upwords are known to exist for all $n \ge 2$. 
	
	\begin{theorem}[Theorems 9, 10, and 13 in \cite{CKMS17}]
		For every $n \ge 2$, there is an upword for $\set{0,1}^n$ that contains exactly one $\diam$. For every $n \ge 4$, there is an upword for $\set{0,1}^n$ that contains exactly two $\diam$'s.
	\end{theorem}
	
	Chen et al. further conjectured that for every $n \ge 2$, and any given position $k$ in the upword, there exists an upword for $\set{0,1}^n$ containing exactly one $\diam$ at position $k$, with a few known exceptions. They verified the conjecture computationally for $2 \le n \le 13$ and provided tables of examples of binary upwords with a single $\diam$. Proposition \ref{prop:gupword threshold} yields a gupword from each of these known upwords.
	
	\begin{example} Applying Proposition \ref{prop:gupword threshold} to the upword $01100\diam011110100$ for $\set{0,1}^4$ produces a gupword for $5$-vertex threshold graphs as shown in Figure \ref{fig:5threshold1}.
		
		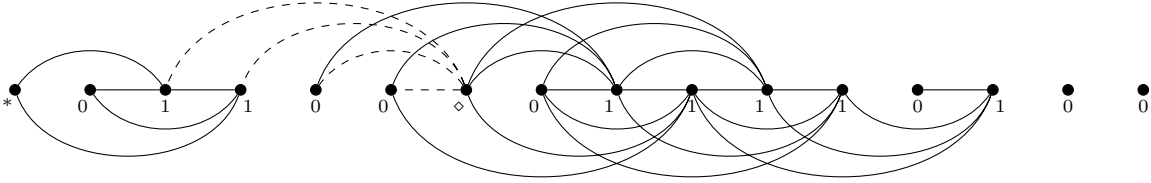
\begin{figure}[h!]
			\centering
			
			\tikzstyle{every node}=[circle, draw, fill=black, inner sep=0pt, minimum width=4pt]
			\begin{tikzpicture}
				
				\node[label={[shift={(-.1,-0.4)}]\scriptsize{*}}] (1) {};
				\node[label={[shift={(-.1,-0.4)}]\scriptsize{0}}] (2) [right of=1] {};
				\node[label={[shift={(0,-0.4)}]\scriptsize{1}}] (3) [right of=2] {};
				\node[label={[shift={(.1,-0.4)}]\scriptsize{1}}] (4) [right of=3] {};
				\node[label={[shift={(0,-0.4)}]\scriptsize{0}}] (5) [right of=4] {};
				\node[label={[shift={(-.1,-0.4)}]\scriptsize{0}}] (6) [right of=5] {};
				\node[label={[shift={(-.1,-0.4)}]\scriptsize{$\diam$}}] (7) [right of=6] {};
				\node[label={[shift={(-.1,-0.4)}]\scriptsize{0}}] (8) [right of=7] {};
				\node[label={[shift={(-.1,-0.4)}]\scriptsize{1}}] (9) [right of=8] {};
				\node[label={[shift={(0,-0.4)}]\scriptsize{1}}] (10) [right of=9] {};
				\node[label={[shift={(-.1,-0.4)}]\scriptsize{1}}] (11) [right of=10] {};
				\node[label={[shift={(0,-0.4)}]\scriptsize{1}}] (12) [right of=11] {};
				\node[label={[shift={(0,-0.4)}]\scriptsize{0}}] (13) [right of=12] {};
				\node[label={[shift={(.1,-0.4)}]\scriptsize{1}}] (14) [right of=13] {};
				\node[label={[shift={(0,-0.4)}]\scriptsize{0}}] (15) [right of=14] {};
				\node[label={[shift={(0,-0.4)}]\scriptsize{0}}] (16) [right of=15] {};
				\draw
				(4) to (3) to (2)
				(12) to (11) to (10) to (9) to (8)
				(14) to (13)
				
				(3) to[bend left=-55] (1)
				(4) to[bend left=55] (2)
				(4) to[bend left=70] (1)
				
				(9) to[bend left=-55] (7)
				(9) to[bend left=-70] (6)
				(9) to[bend left=-70] (5)
				
				(10) to[bend left=55] (8)
				(10) to[bend left=70] (7)
				(10) to[bend left=70] (6)
				
				(11) to[bend left=-55] (9)
				(11) to[bend left=-70] (8)
				(11) to[bend left=-70] (7)
				
				(12) to[bend left=70] (9)
				(12) to[bend left=70] (8)
				(12) to[bend left=55] (10)
				
				(14) to[bend left=55] (12)
				(14) to[bend left=70] (11)
				(14) to[bend left=70] (10);
				
				\draw[dashed]
				(7) to (6)
				(7) to[bend left=-55] (5)
				(7) to[bend left=-70] (4)
				(7) to[bend left=-70] (3);
				
			\end{tikzpicture}   
			
			\caption{Gupword corresponding to the upword $01100\diam011110100$ for $\set{0,1}^4$.}
			\label{fig:5threshold1}
		\end{figure}
	\end{example}
	
	The following theorem also gives a general construction of a binary upword for every $n$.
	
	\begin{theorem}[Theorem 17 in \cite{CKMS17}]\label{thm:17}
		For every $n \ge 2$, $\diamond^{n-1}01^n$ is an upword for $\set{0,1}^{n}$.
	\end{theorem}
	
	We can use Theorem \ref{thm:17} to construct a gupword for $n$-vertex threshold graphs for any $n \ge 3$.
	
	\begin{example}
		Figure \ref{fig:5threshold} shows a gupword for $5$-vertex threshold graphs constructed from the upword $\diam \diam \diam 01111$, which is an upword for $\set{0,1}^{4}$ by Theorem \ref{thm:17}.
		
		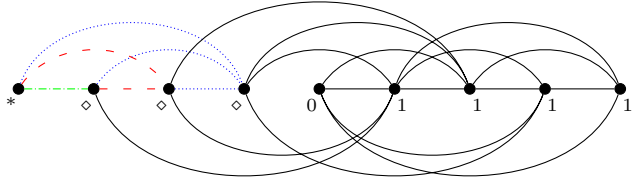
\begin{figure}[h!]
			\centering
			\tikzstyle{every node}=[circle, draw, fill=black, inner sep=0pt, minimum width=4pt]
			\begin{tikzpicture}
				
				\node[label={[shift={(-.1,-0.4)}]\scriptsize{*}}] (1) {};
				\node[label={[shift={(-.1,-0.4)}]\scriptsize{$\diamond$}}] (2) [right of=1] {};
				\node[label={[shift={(-.1,-0.4)}]\scriptsize{$\diamond$}}] (3) [right of=2] {};
				\node[label={[shift={(-.1,-0.4)}]\scriptsize{$\diamond$}}] (4) [right of=3] {};
				\node[label={[shift={(-.1,-0.4)}]\scriptsize{0}}] (5) [right of=4] {};
				\node[label={[shift={(.1,-0.4)}]\scriptsize{1}}] (6) [right of=5] {};
				\node[label={[shift={(.1,-0.4)}]\scriptsize{1}}] (7) [right of=6] {};
				\node[label={[shift={(.1,-0.4)}]\scriptsize{1}}] (8) [right of=7] {};
				\node[label={[shift={(.1,-0.4)}]\scriptsize{1}}] (9) [right of=8] {};
				
				\draw
				(5) to (6)
				(6) to (7)
				(7) to (8)
				(8) to (9);
				
				\draw[densely dashdotted, green]
				(1) to (2);
				
				\draw[densely dotted, blue]
				(1) to[bend left=70] (4)
				(3) to (4)
				(2) to[bend left=55] (4);
				
				\draw
				(2) to[bend left=-70] (6)
				(3) to[bend left=-70] (6)
				(3) to[bend left=70] (7)
				(4) to[bend left=55] (6)
				(4) to[bend left=70] (7)
				(4) to[bend left=-70] (8)
				(5) to[bend left=55] (7)
				(5) to[bend left=-70] (8)
				(5) to[bend left=-70] (9)
				(6) to[bend left=55] (8) 
				(6) to[bend left=70] (9)
				(7) to[bend left=55] (9);
				
				\draw[loosely dashed, red]
				(1) to[bend left=55] (3)
				(2) to (3);
				
			\end{tikzpicture}
			\caption{The gupword for $5$-vertex threshold graphs corresponding to $\diam \diam \diam 01111$. The diamond edges of the same color are thought of as all present or all absent.}
			\label{fig:5threshold}
		\end{figure}
	\end{example}
	
	\begin{prop}\label{prop:thresholdgupcycle} For every upcycle for $\set{0,1}^{n-1}$ of length greater than $2(n-1)$, there exists a gupcycle for $n$-vertex threshold graphs of the same length.
	\end{prop}
	
	The proof of Proposition \ref{prop:thresholdgupcycle} follows similarly to that of Proposition \ref{prop:gupword threshold}. The main distinction is the partial graph $G$ is on vertex set $\{1,\dots,|w|\}$ and all addition is done $\mod{|w|}$. The length restriction ensures that the cycling gluing does not change the $n$-windows, as shown in Lemma \ref{lem:cyclicgluing}.
	
	Question 6.2 of \cite{G18} asks whether all upcycles have exactly one $\diam$ in each $n$-window. If so, that would imply that every binary upcycle for $\set{0,1}^{n-1}$ has length $2^{n-2}$, and so the length condition of Proposition \ref{prop:thresholdgupcycle} would be satisfied for all $n \ge 6$. By Proposition 4.7 of \cite{G18}, the length of every binary upcycle for $\set{0,1}^{n-1}$ is greater than $2^{\sqrt{(n-1)-7/4}+1/2}$, which implies that the length condition of Proposition \ref{prop:thresholdgupcycle} is satisfied for all $n \ge 33$.
	
	The known examples of upcycles for $\set{0,1}^{n-1}$ all have either $n=5$ (in which case the upcycle length must be $8$, and the length condition of Proposition \ref{prop:thresholdgupcycle} is not satisfied) or $n=9$ and length $128$. Binary upcycles of the latter type are shown to exist and listed in \cite{FGKMM21}. By Proposition \ref{prop:thresholdgupcycle} they can be used to construct gupcycles for $9$-vertex threshold graphs. These gupcycles have $16$ diamond edge sets. At $128$ vertices, they are too long to be drawn here.
	
	\section{Permutations}\label{sec:perms}
	
	Universal cycles and their variants for permutations have been widely studied. In this section we show that these results translate nicely to the setting of graph universal cycles for labeled permutation graphs. We will show a bijection explaining the connection in Section \ref{subsec:permbijn}. Then we will extend this correspondence in two ways: in Section \ref{subsec:permgupcycle}, to compressed forms of both universal cycles and graph universal cycles, which represent every permutation (graph) even more compactly by covering more than one element per window, and in Section \ref{subsec:socycle}, to $s$-overlap forms of both. Section \ref{subsec:permfuture} demonstrates new ways to shorten universal words, which also correspond to graph universal words.
	
	\subsection{Introduction to universal cycles and graph universal cycles for permutations}\label{subsec:permintro}
	
	\subsubsection{Permutations, order-isomorphism, and universal cycles for permutations} 
	We will write $S_n(A)$ for the set of ordered $n$-tuples of distinct elements of a totally ordered set $A$. Thus $S_n = S_n([n])$.
	Two permutations $\sigma, \tau \in S_n(A)$ are \emph{order-isomorphic}, denoted by $\sigma \cong \tau$, if, for every $i, j \in [n]$, we have $\sigma(i) < \sigma(j)$ if and only if $\tau(i) < \tau(j)$. Each permutation $\pi \in S_n(A)$ is order-isomorphic to some permutation in $S_n$, which we denote by $\red(\pi)$. Equivalently two permutations are order-isomorphic if $\red(\sigma) = \red(\tau)$. We often consider permutations as words and write $\sigma_i$ for $\sigma(i)$.
	
	Universal cycles for $n$-permutations are cyclic sequences of numbers in which every permutation on $[n]$ is order-isomorphic to exactly one $n$-window; see~\cite{CDG92}. For example, $124324$ is a universal cycle for $3$-permutations because its $3$-windows, from left to right, are order-isomorphic to the permutations $123$, $132$, $321$, $213$, $231$, and $312$. Universal cycles for $n$-permutations exist for every $n$ and in \cite{J09} were shown to exist on $n+1$ distinct symbols.
	
	Universal cycles for permutations have been studied using the clustered graph of overlapping permutations, an analogue of the De Bruijn graph. We define a more general version, where the permutations overlap by $s$, to use in Section \ref{subsec:socycle}.
	
	\begin{definition}\label{def:permoverlap}
		An ordered pair of permutations $(\sigma, \tau) \in S_n^2$ \emph{overlaps by $s$} if $\sigma_{n-s+1}\cdots\sigma_n \cong \tau_1\cdots\tau_s$. An ordered pair of permutations $(\sigma, \tau)$ \emph{overlaps} if it overlaps by $n-1$.
	\end{definition}
	
	\begin{definition}
		The \emph{clustered graph of $s$-overlapping $n$-permutations}, denoted by $\mathcal{O}(n,s)$, is a directed multigraph on vertex set $S_{s}$. For each $\pi \in S_n$, there is a directed edge $(\red(\pi_1\cdots\pi_{s}),\red(\pi_{n-s+1}\cdots\pi_{n}))$. Thus we identify the edge set with $S_n$. The \emph{clustered graph of overlapping $n$-permutations}, denoted by $\mathcal{O}(n)$, is $\mathcal{O}(n,n-1)$, the graph of $(n-1)$-overlapping $n$-permutations. 
	\end{definition}

	\begin{figure}[h!]
		\centering
		\begin{tikzpicture}
			\node[circle, draw, thick, fill=white, inner sep=0pt, minimum width=30pt] (0) at (.25,0) {12};
			\node[circle, draw, thick, fill=white, inner sep=0pt, minimum width=30pt] (00) at (3.25,0) {21};

			\draw[->, -stealth, thick] (00) to [bend right=-40] (0);
			\draw[->, -stealth, thick, bend left=60]
			(00) to (0);
			\draw[->, -stealth, thick, bend left=40]
			(0) to (00);
			\draw[->, -stealth, thick, bend left=60]
			(0) to (00);

			\draw[thick, ->, -stealth] (0) to [out=155,in=205,looseness=7] (0);
			
			\draw[thick, ->, -stealth] (00) to [out=-25,in=25,looseness=7] (00);

			\node[white, label={\tiny 231}] (a) at (1.75, 1) {};
			
			\node[white, label={\tiny 312}] (d) at (1.75, -1.6) {};
			
			\node[white, label={\tiny 132}] (k) at (1.75, .25) {};
			
			\node[white, label={\tiny 213}] (h) at (1.75, -.9) {};
			
			\node[white, label={\tiny 123}] (n) at (-1.4, -.3) {};
			
			\node[white, label={\tiny 321}] (w) at (4.9, -.3) {};
		\end{tikzpicture}
		\caption{The clustered graph of overlapping $3$-permutations.}
		\label{fig:permoverlap}
	\end{figure}
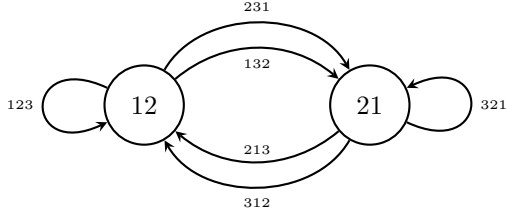
	
	For $\sigma, \tau \in S_n$, $(\sigma,\tau)$ is a path in the clustered graph of $s$-overlapping permutations if and only if $(\sigma, \tau)$ overlaps by $s$. Given a universal cycle for $n$-permutations, its sequence of $n$-windows gives a cyclic ordering of $S_n$ such that every consecutive pair overlaps. This ordering in turn can be viewed as an Euler tour in the clustered graph of overlapping permutations that covers $S_n$ in the same order. However, the following conjecture remains open.
	
	\begin{conjecture}[Conjecture 3.1.3 in \cite{H90}, \cite{CDG92}]\label{conj:cycleclosing}
		Given any Euler tour of the clustered graph of overlapping permutations (i.e., an ordering of $S_n$ such that each consecutive pair of permutations overlaps), there is a universal cycle for permutations that covers $S_n$ in the same order.
	\end{conjecture}
	
	\begin{remark}\label{rem:tourtoword}
		It is known \cite{CDG92, KPV19} that given any Euler tour (or Euler trail) of the clustered graph of overlapping permutations, there is a universal \emph{word} for permutations that covers $S_n$ in the same order. The difficulty in the conjecture is in closing the cycle. In other words, it is not known whether the first and last $n-1$ symbols of a universal word for $n$-permutations can always be made to be identical, or if they sometimes must take on different, order-isomorphic forms.
	\end{remark}
	
	In contrast, in Section \ref{subsec:permbijn} we show that given any Euler tour, there is a graph universal cycle for permutation graphs that covers $S_n$ in the same order. One implication of our work is that if Conjecture \ref{conj:cycleclosing} is false, then there is a graph universal cycle for permutations that covers the permutations in an order that cannot be achieved by any universal cycle for permutations.
	
	\subsubsection{Graph universal cycles for permutation graphs}\label{sec:permgucycles}
	
	Permutations of $[n]$ can equivalently be represented by \emph{permutation graphs} on vertex set $[n]$. The permutation graph $G(\pi)$ for $\pi \in S_n$ contains the edge $\set{i,j}$, where without loss of generality $i < j$, if and only if $\pi(i) > \pi(j)$. In other words, the edges indicate the inversions of $\pi$.
	
	In \cite{CGGP21}, Cantwell et al. proved the existence of gucycles for permutation graphs on vertex set $[n]$ for every $n$ by defining the following arc digraph and showing that it is Eulerian. 
	
	\begin{definition}
		The \emph{$s$-overlap arc digraph for permutation graphs on $[n]$}, denoted by $P_{n,s}$, is the directed multigraph whose vertices are the permutation graphs on vertex set $[s]$ and whose edge set is the set of permutation graphs on vertex set $[n]$. An edge $e$ points from $u$ to $v$ if and only if $e[1,2,\ldots,s] \oiso u$ and $e[n-s+1, n-s+2, \ldots, n] \oiso v$.
		The \emph{arc digraph for permutation graphs on $[n]$}, denoted by $P_n$, is the $(n-1)$-overlap arc digraph for permutation graphs on $[n]$. Equivalently, $P_n$ is the subgraph of $D_n$ induced by the edges that are the permutation graphs on $[n]$.
	\end{definition}
	
	Cantwell et al. \cite{CGGP21} noted that the permutations of $[n]$ are in bijection with the permutation graphs on vertex set $[n]$, so these gucycles may be understood as representing each permutation of $[n]$ exactly once.

	\subsection{Bijection between Euler tours of the clustered graph of overlapping permutations and graph universal cycles for permutation graphs} \label{subsec:permbijn}
	
	We will demonstrate a bijection between the Euler tours of the clustered graph of overlapping permutations and the graph universal cycles for permutation graphs, which shows that universal cycles for permutations and gucycles for permutation graphs are more closely related than previously known. Given a universal cycle for $n$-permutations, we construct a gucycle for permutation graphs on $[n]$ that covers $S_n$ in the same order. Given a gucycle for permutation graphs, we construct an Euler tour of the clustered graph of overlapping permutations, which yields a universal word for $n$-permutations and---if Conjecture \ref{conj:cycleclosing} holds---also yields a universal cycle for $n$-permutations.
	
	\begin{example} The gucycle in Figure \ref{fig:3permguc} corresponds to the Euler tour of the clustered graph of overlapping permutations given by $123,231,321,312,132,213$.

		\begin{figure}[h!]
			\centering
			\tikzstyle{every node}=[circle, draw, fill=black, inner sep=0pt, minimum width=4pt]
			\begin{tikzpicture}[scale=0.4]
				
				\node[label={[shift={(0,-0.4)}]\scriptsize{1}}] (1) {};
				\node[label={[shift={(0,-0.4)}]\scriptsize{2}}] (2) [right of=1] {};
				\node[label={[shift={(0,-0.4)}]\scriptsize{3}}] (3) [right of=2] {};
				\node[label={[shift={(0,-0.4)}]\scriptsize{4}}] (4) [right of=3] {};
				\node[label={[shift={(0,-0.4)}]\scriptsize{5}}] (5) [right of=4] {};
				\node[label={[shift={(0,-0.4)}]\scriptsize{6}}] (6) [right of=5] {};
				
				\path[every node/.style]
				(3) edge node {} (4)
				(4) edge node {} (5);
				
				\draw
				(1) to[bend left=70] (6)
				(2) to[bend left=55] (4)
				(3) to[bend left=55] (5)
				(4) to[bend left=55] (6);
			\end{tikzpicture}
			\caption{A graph universal cycle for permutation graphs with $n=3$, which covers the $3$-permutations in the same order as the universal cycle for permutations $256413$.}
			\label{fig:3permguc}
		\end{figure}
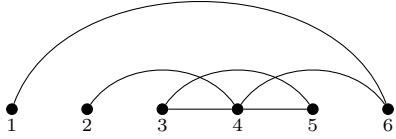
	\end{example}
	
	In contrast to \cite{CGGP21}, we will show in Lemma \ref{lem:soverlap} that permutation overlaps, as defined in Definition \ref{def:permoverlap}, correspond exactly to graph overlaps, as defined in Definition \ref{def:partialgraphoverlap}. We write permutations in the standard one-line notation so that, for example, $\pi = 312 \in S_3$ is the permutation defined by $\pi(1) = 3$, $\pi(2) = 1$, and $\pi(3) = 2$. Following the definition of permutation graphs in Section \ref{sec:permgucycles}, $\set{1,2}$ and $\set{1,3}$ are the inversions of $312$, so they are the edges of the permutation graph $G(312)$, shown in Figure \ref{fig:312}. In \cite{CGGP21} the stated definition of permutation graphs is the same, but Figures 7 and 8 in \cite{CGGP21} show that either the permutation graph of $\sigma \in S_n$ records the inversions of $\sigma^{-1}$ instead of the inversions of $\sigma$, or the vertices of the permutation graph are labeled in the order of $\sigma$ instead of the order $1,2,\ldots, n$. As a result, for example, the permutation graphs $G(312)$ and $G(231)$ appear to be swapped, and, more importantly, the authors found that graph overlaps did not necessarily correspond with permutation overlaps.
	
	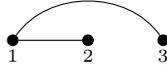
\begin{figure}[h!]
		\centering
		\tikzstyle{every node}=[circle, draw, fill=black, inner sep=0pt, minimum width=4pt]
		\centering
		\begin{tikzpicture}[scale=0.4]
			
			\node[label={[shift={(0,-0.4)}]\scriptsize{1}}] (1) {};
			\node[label={[shift={(0,-0.4)}]\scriptsize{2}}] (2) [right of=1] {};
			\node[label={[shift={(0,-0.4)}]\scriptsize{3}}] (3) [right of=2] {};
			
			\path[every node/.style]
			(1) edge node {} (2);
			
			\draw
			(1) to[bend left=55] (3);
		\end{tikzpicture}
		\caption{The permutation graph $G(312)$}
		\label{fig:312}
	\end{figure}
	
	\begin{lemma}\label{lem:soverlap}
		For every $n \ge 2$ and $1 \le s \le n-1$, the ordered pair of permutations $(\sigma, \tau) \in S_n^2$ overlaps by $s$ if and only if $(G(\sigma),G(\tau))$ overlaps by $s$.
	\end{lemma}
	
	\begin{proof}
		By Definition \ref{def:permoverlap}, $(\sigma, \tau) \in S_n^2$ overlaps by $s$ if and only if $\sigma_{n-s+1}\cdots\sigma_n \cong \tau_1\cdots\tau_s$. Let $i$ and $j$ satisfy $n-s+1 \le i < j \le n$. By the definition of order isomorphism, we have $\sigma_i > \sigma_j$ if and only if $\tau_{i-n+s} > \tau_{j-n+s}$. By the definition of a permutation graph, we have $\sigma_i > \sigma_j$ if and only if $i$ and $j$ are adjacent in $G(\sigma)$, and $\tau_{i-n+s} > \tau_{j-n+s}$ if and only if $i-n+s$ and $j-n+s$ are adjacent in $G(\tau)$. Therefore, $i$ and $j$ are adjacent in $G(\sigma)$ if and only if $i-n+s$ and $j-n+s$ are adjacent in $G(\tau)$, so $G(\sigma)[n-s+1, \ldots, n] \oiso G(\tau)[1,\ldots,s]$.
	\end{proof}

	\begin{lemma}\label{lem:isoperm} For every $n$ and every $s$, the clustered graph of $s$-overlapping permutations $\mathcal{O}(n,s)$ is isomorphic to $P_{n,s}$, the arc digraph for permutation graphs on $[n]$. The isomorphism is given by the function $G$ that sends each permutation to its permutation graph, applied to both the vertex and edge sets. The function $G$ also induces a bijection between the Euler tours of $\mathcal{O}(n,s)$ and the Euler tours of $P_{n,s}$ that preserves the order in which the permutations are covered.
	\end{lemma}
	
	\begin{proof}The pair of functions $(G,G)$ is an isomorphism of  $\mathcal{O}(n,s)$ and $P_{n,s}$ by Lemma \ref{lem:soverlap}, as both directed multigraphs are defined in the corresponding ways based on overlaps in permutations and permutation graphs. The last sentence of the lemma follows from the definition of isomorphism.
	\end{proof}
	
	\begin{theorem}\label{thm:perm graph bij}For every $n \ge 3$ there is a bijection between Euler tours of the clustered graph of overlapping permutations and graph universal cycles for permutation graphs that preserves the order in which $S_n$ is covered.\end{theorem}
	
	\begin{proof} By Lemma \ref{lem:isoperm}, the isomorphism $G$ induces a bijection between Euler tours of $\mathcal{O}(n)$ and Euler tours of $P_n$ that preserves the order in which the permutations are covered.
		
		The Euler tours of $P_n$ also are in bijection with the gucycles for permutation graphs on $[n]$ in a way that preserves the order in which the permutations are covered: an Euler tour lists the permutation graphs on $[n]$ in an order $p_1,p_2,\ldots, p_{n!}$ such that consecutive permutation graphs overlap. Let $q_1 = p_1$, and for all $2 \le k \le n!$ let $q_k$ be the gluing of $(q_{k-1},p_k)$, which is possible by Lemma \ref{lem:gluing}. Then $q_{n!}$ is a guword $H$, which has length $n!+n-1$, and since $(p_{n!},p_1)$ overlaps, we have that $(H,H)$ overlaps by $n-1$. Since the gluing operation does not create edges at distance greater than $n-1$, and $n \ge 3$, $H$ has no edges with one vertex in $\set{1,2,\ldots, n-1}$ and the other in $\set{n!+1, n!+2, \ldots, n!+n-1}$. For every $n \ge 3$, with $m = n!+n-1$, we have $m-n+1 = n! > 2(n-1)$, so Lemma \ref{lem:cyclicgluing} implies the existence of a cyclic gluing $J$ of $H$. Since $J$ and $H$ have the same $n$-windows in the same order, $J$ is a gucycle covering $S_n$ in the same order as the Euler tour.
		
		Given a gucycle $H$, the corresponding Euler tour of $P_{n}$  is $(p_1,p_2, \ldots, p_{n!})$, where for all $i \in [n!]$, $p_i := H[i, i+1, \ldots, i+n-1]$. That is, it is obtained by taking the $n$-windows of $H$ in order.
	\end{proof}
	
	\begin{remark}Theorem \ref{thm:perm graph bij} and its proof show how to convert a given universal word (or cycle) for permutations to a gucycle for permutation graphs, and a given gucycle for permutation graphs to at least a universal word for permutations, if not a universal cycle for permutations (see Conjecture \ref{conj:cycleclosing} and Remark \ref{rem:tourtoword}).\end{remark}
	
	\subsection{Correspondence between shortened universal words for permutations and gupcycles for permutation graphs}\label{subsec:permgupcycle}
	
	In \cite{KPV19}, shortened universal words for permutations are obtained by introducing incomparable elements at distance $n-1$. For example, $3123$ covers the permutations $3124$ and $4123$. In general, a word $w$ over the natural numbers defines a partial order on the positions of the word, and is said to \emph{cover} all permutations that represent total orders extending the partial order $w$. In \cite{KLSS22}, it is shown that shortened universal cycles for permutations also exist for the corresponding lengths, but we will not rely on this result and instead connect shortened universal words for permutations with graph universal partial cycles for permutation graphs.
	
	\begin{theorem}[Theorem 7 in \cite{KPV19}]\label{thm:universal words for perms}
		Using incomparable elements at distance $n-1$, one can obtain a shortened universal word for $n$-permutations of length $n! + n-1 - i(n-1)$ for all $i$ in $0 \le i \le (n-2)!$.
	\end{theorem}
	
	\begin{remark}\label{rem:eulertour}Every Euler trail in an Eulerian directed multigraph can be extended to an Euler tour. Thus, these shortened universal words for $n$-permutations correspond not only to Euler trails but also to Euler tours of partially compressed clustered graphs of overlapping $n$-permutations. In fact, in \cite{KPV19}, Theorem \ref{thm:universal words for perms} is proved by first showing the existence of such Euler tours. The Euler tours have lengths $n! - i(n-1)$ for all $i$ in $0 \le i \le (n-2)!$, which are the numbers of $n$-windows in the shortened universal words.
	\end{remark}
	
	In \cite{KPV19} it is observed that if $e$ and $f$ are twin edges in the clustered graph of overlapping $n$-permutations then (without loss of generality) $e_1 = k = f_n$ and $e_n = k+1 = f_1$, for some $k \in [n-1]$, and $e_2\cdots e_{n-1} = f_2\cdots f_{n-1}$. Therefore $e$ and $f$ can be \emph{compressed} into a single edge $g$ that has $g_1 = k = g_n$ and $g_2\cdots g_{n-1} \iso e_2\cdots e_{n-1}$. In this way, $g$ covers $e$ and $f$ but no other $n$-permutations. To use this idea to create graph universal partial cycles, we observe that the permutation graphs of twin edges $e$ and $f$, $G(e)$ and $G(f)$, differ only by the presence or absence of the edge $\set{1,n}$, so the word $g$ can equivalently be represented by the partial graph having the diamond edge $\set{1,n}$ and otherwise agreeing with both $G(e)$ and $G(f)$.
	
	\begin{remark}\label{rem:compressediso}
		By Lemma \ref{lem:isoperm}, compressing twins in the clustered graph of overlapping $n$-permutations is equivalent to compressing twins in the arc digraph $P_n$. Thus, there is a bijection between the Euler tours of partially compressed clustered graphs of overlapping $n$-permutations and the Euler tours of partially compressed arc digraphs for permutation graphs on $[n]$.
	\end{remark}
	
	In the proof of Theorem 7 in \cite{KPV19}, it is shown that the clustered graph of overlapping permutations contains $(n-2)!$ disjoint cycles formed by twin edges, each of length $n-1$ by \cite[Lemma 6]{KPV19}, and that, after compressing any number of these cycles of twin edges, the clustered graph of overlapping permutations remains Eulerian.
	
	\begin{theorem}\label{thm:permgupbij}
		For every $n$, there is a bijection between the set of Euler tours of length greater than $2(n-1)$ of partially compressed clustered graphs of overlapping $n$-permutations, and the set of graph universal partial cycles for permutation graphs on $[n]$ in which every diamond edge is at distance $n-1$, which preserves the partial order in which $S_n$ is covered.
	\end{theorem}
	
	\begin{proof}By Remark \ref{rem:compressediso}, we need only show a bijection between the set of Euler tours of length greater than $2(n-1)$ of partially compressed arc digraphs for permutation graphs on $[n]$, and the set of graph universal partial cycles for permutation graphs on $[n]$ in which every diamond edge is at distance $n-1$, which preserves the partial order in which $S_n$ is covered.
		
		Let $T = (p_1, p_2, \ldots, p_\ell)$ be an Euler tour of a partially compressed arc digraph for permutation graphs on $[n]$ with $\ell > 2(n-1)$. The $p_i$'s are ordered partial graphs containing one or more permutation graphs on $[n]$, and $(p_i, p_{i+1})$ overlaps for every $i \in[\ell]$, where $p_{\ell+1}:=p_1$. Let $q_1 = p_1$, and for all $2 \le k \le \ell$ let $q_k$ be the gluing of $(q_{k-1},p_k)$. Then $q_{\ell}$ is a gupword $H$, which has length $\ell+n-1$, and since $(p_{\ell},p_1)$ overlaps, we have that $(H,H)$ overlaps by $n-1$. The condition $\ell > 2(n-1)$ implies $(\ell+1)-(n-1) > n-1$; since the gluing operation does not create edges (or diamond edges) at distance greater than $n-1$, there are no edges (or diamond edges) with one vertex in $\set{1,2,\ldots,n-1}$ and the other in $\set{\ell+1,\ell+2,\ldots,\ell+n-1}$. By Lemma \ref{lem:cyclicgluing} with $m = \ell+n-1$, there exists a cyclic gluing $J$ of $H$ that has the same $n$-windows as $H$ and in the same order. Thus $J$ is a gupcycle covering $S_n$ in the same partial order as in $H$ and as in $T$. 
		
		For the inverse, let $J$ be a gupcycle for permutation graphs on $[n]$ in which every diamond edge is at distance $n-1$. Let $\ell$ be the number of vertices of $J$, and, for all $i\in [\ell]$, let $J_i = J[i,i+1,\ldots, i+n-1]$, the $i$th $n$-window of $J$. In $P_n$, compress exactly the pairs of twin edges representing permutation graphs contained in the same window $J_i$. Then $(J_1,J_2,\ldots,J_\ell)$ is the corresponding Euler tour of this partially compressed arc digraph.
	\end{proof}

	\begin{corollary}
		For every $n \ge 4$, using diamond edges at distance $n-1$, one can obtain a graph universal partial cycle for permutation graphs on $[n]$ of length $n! - i(n-1)$ for all $i$ in $0 \le i \le (n-2)!$.
	\end{corollary}
	
	\begin{proof}Theorem \ref{thm:universal words for perms} guarantees the existence of shortened universal words for $n$-permutations, which by Remark \ref{rem:eulertour} correspond to Euler tours of partially compressed clustered graphs of overlapping $n$-permutations, of lengths $n! - i(n-1)$ for all $i$ in $0 \le i \le (n-2)!$. For every $n \ge 4$, these Euler tours' lengths are greater than $2(n-1)$ because 
		\[
		n! - i(n-1) \ge n! - (n-1)! = (n-1)(n-1)! > 2(n-1).
		\]
		Theorem \ref{thm:permgupbij} converts these Euler tours into graph universal partial cycles of the same lengths.
	\end{proof}

	In Theorem \ref{thm:permgupbij} we specify that every diamond edge is at distance $n-1$ in order to match the operation of compressing twins. See Section \ref{subsec:permfuture} for other kinds of gupwords for permutation graphs. More generally, if $w$ is a shortened universal word (or cycle) for $n$-permutations, then a graph universal partial word (or cycle) can be formed by gluing together ordered partial graphs determined by the $n$-windows of $w$.

	\begin{example}Figure \ref{fig:permgupcycle} shows an example of a gupcycle for permutation graphs on $[4]$ and its corresponding Euler tour.
		
		\begin{figure}[h!]
			\centering
			
			\tikzstyle{every node}=[circle, draw, fill=black, inner sep=0pt, minimum width=4pt, scale=0.9]
			\begin{tikzpicture}
				
				\node[label={[shift={(0,-0.4)}]\scriptsize{1}}] (16) {};
				\node[label={[shift={(0,-0.4)}]\scriptsize{2}}] (17) [right of=16] {};
				\node[label={[shift={(0,-0.4)}]\scriptsize{3}}] (18) [right of=17] {};
				\node[label={[shift={(0,-0.4)}]\scriptsize{4}}] (5) [right of=18] {};
				\node[label={[shift={(0,-0.4)}]\scriptsize{5}}] (6) [right of=5] {};
				\node[label={[shift={(0,-0.4)}]\scriptsize{6}}] (7) [right of=6] {};
				\node[label={[shift={(0,-0.4)}]\scriptsize{7}}] (8) [right of=7] {};
				\node[label={[shift={(0,-0.4)}]\scriptsize{8}}] (9) [right of=8] {};
				\node[label={[shift={(0,-0.4)}]\scriptsize{9}}] (10) [right of=9] {};
				\node[label={[shift={(0,-0.4)}]\scriptsize{10}}] (11) [right of=10] {};
				\node[label={[shift={(0,-0.4)}]\scriptsize{11}}] (12) [right of=11] {};
				\node[label={[shift={(0,-0.4)}]\scriptsize{12}}] (13) [right of=12] {};
				\node[label={[shift={(0,-0.4)}]\scriptsize{13}}] (14) [right of=13] {};
				\node[label={[shift={(0,-0.4)}]\scriptsize{14}}] (15) [right of=14] {};  \node[label={[shift={(0,-0.4)}]\scriptsize{15}}] (1) [right of=15] {};
				\node[label={[shift={(0,-0.4)}]\scriptsize{16}}] (2) [right of=1] {};
				\node[label={[shift={(0,-0.4)}]\scriptsize{17}}] (3) [right of=2] {};
				\node[label={[shift={(0,-0.4)}]\scriptsize{18}}] (4) [right of=3] {};
				
				\draw
				(1) to (2)
				(5) to (6)
				(7) to (8)
				(12) to (11) to (10) to (9)
				(13) to (14)
				(16) to (17)
				
				(3) to[bend left=-55] (1)
				(7) to[bend left=-55] (5)
				(8) to[bend left=-70] (5)
				(11) to[bend left=-55] (9)
				(12) to[bend left=-55] (10)
				(14) to[bend left=-55] (12)
				(12) to[bend left=-70] (9)
				(14) to[bend left=-70] (11)
				(18) to[bend left=-55] (16)
				(15) to[bend left=55] (2)
				
				(5) to[bend left=-70] (16)
				(5) to[bend left=-55] (17)
				(5) to (18)
				(6) to[bend left=-70] (17)
				(7) to[bend left=-70] (18)
				(6) to[bend left=-55] (18);
				
				\draw[dashed]
				(4) to[bend left=-70] (1)
				(10) to[bend left=-70] (7)
				(13) to[bend left=-70] (10)
				(14) to[bend left=70] (2)
				(15) to[bend left=70] (3)
				(8) to[bend left=70] (11);
				
			\end{tikzpicture}   
			
			\caption{The gupcycle for the Euler tour ($4231$, $3421$, $4312$, $4132$, $1324$, $2132$, $1431$, $4321$, $3213$, $3241$, $2314$, $2134$, $1231$, $2312$, $3123$, $1234$, $1243$, $1423$) of the arc digraph for permutation graphs on $[4]$ after compressing twins to obtain the edges $2132$, $1431$, $3213$, $1231$, $2312$, $3123$.}
			\label{fig:permgupcycle}
		\end{figure}
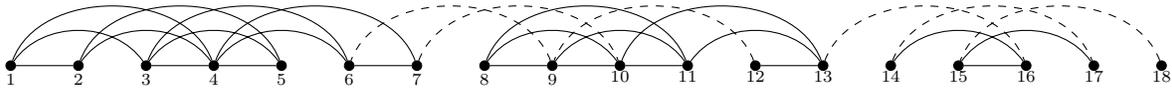
	\end{example}
	
	\subsection{Bijection between Euler tours of the clustered graph of $s$-overlapping permutations and $s$-gocycles for permutation graphs}\label{subsec:socycle}
	
	In this section we consider the generalization of universal cycles to $s$-overlap cycles, first introduced in \cite{GKN10}, where consecutive elements overlap by a constant $s \le n-1$. Since $n+1$ distinct symbols are needed in a universal cycle for $n$-permutations, in \cite{HH13perm}, the authors showed the existence of $s$-overlap cycles for permutations of $[n]$ that use only $n$ distinct symbols, with $s \le n-2$. Here we generalize their definition of an $s$-overlap cycle for permutations to use order-isomorphism rather than equality as words. We do not restrict the number of distinct symbols.
	
	\begin{definition}\label{def:socycle}
		An \emph{$s$-overlap cycle} (or \emph{$s$-ocycle}) for permutations on $[n]$ is a cyclic sequence of numbers in which every permutation on $[n]$ is order-isomorphic to exactly one $n$-window whose index set is equivalent to $\set{(i-1)(n-s)+1, (i-1)(n-s)+2, \ldots, (i-1)(n-s)+n} \mod{n!(n-s)}$ for some $i \in [n!]$.
	\end{definition}
	
	In other words, an $s$-ocycle is a compact representation of a cyclic listing of the permutations on $[n]$ such that every pair of consecutive permutations $(\sigma,\tau)$ overlaps by $s$. (Recall we defined permutation overlapping in Definition \ref{def:permoverlap} such that in an $s$-ocycle the last $s$ letters of each permutation are order-isomorphic to the first $s$ letters of the next permutation.) In particular, by Definition \ref{def:socycle}, an $(n-1)$-ocycle is a universal cycle. The length of an $s$-ocycle for permutations on $[n]$ is $n!(n-s)$.
	
	\begin{example}\label{example:socycle} The cyclic word
		$012436543645362401324675689568732410425380756432$ is a $2$-overlap cycle for permutations on $[4]$ because its $4$-windows in the relevant positions are order-isomorphic to $1234$, $1324$, $1432$, $3214$, $1423$, $2314$, $2413$, $3412$, $1243$, $2134$, $1342$, $3124$, $2341$, $4123$, $2431$, $4213$, $3421$, $2143$, $3142$, $3241$, $4132$, $4231$, $4321$, and $4312$, which are all of the $4$-permutations exactly once.
	\end{example}
	
	We make the following analogous definition in the ordered graph context.
	
	\begin{definition}
		A \emph{graph $s$-overlap cycle} (or \emph{$s$-gocycle}) for permutation graphs on $[n]$ is a cyclically ordered graph $G$ on vertex set $[n!(n-s)]$ in which every permutation graph on vertex set $[n]$ is isomorphic as an ordered graph to exactly one $n$-window of $G$ whose vertex set is equivalent to $\set{(i-1)(n-s)+1, (i-1)(n-s)+2, \ldots, (i-1)(n-s)+n} \mod{n!(n-s)}$ for some $i \in [n!]$.
	\end{definition}
	
	In particular, an $(n-1)$-gocycle is a gucycle.

	\begin{theorem}[Result 4 in Horan and Hurlbert \cite{HH13perm}]\label{thm:HH}
		Let $n, s \in \N$ with $n \ge 2$. If either (1) $1 \le s < n/2$, or (2) $\gcd(s, n) = 1$ with
		$n/2 \le s \le n - 2$, then there exists an $s$-ocycle on the set of permutations of $[n]$, in which the last $s$ letters of each permutation are equal [as words---not merely order-isomorphic] to the first $s$ letters of the next permutation.
	\end{theorem}
	
	Theorem \ref{thm:HH} gives a significant generalization of the existence of universal cycles for $n$-permutations. Note that Theorem \ref{thm:HH} omits the case $s=n-1$ because there is no $(n-1)$-ocycle for permutations of $[n]$ in which the overlaps are equal as words. However, it was previously known that universal cycles for $n$-permutations exist, and by using order-isomorphism we have expanded the definition of $s$-ocycle such that universal cycles for permutations are $(n-1)$-ocycles. Therefore Theorem \ref{thm:HH} can be restated in combination with the case $s=n-1$, using Definition \ref{def:socycle} of $s$-ocycle, as follows.

	\begin{theorem}\label{thm:HHmod}
		Let $n, s \in \N$ with $n \ge 2$. If either (1) $1 \le s < n/2$, or (2) $\gcd(s, n) = 1$ with
		$n/2 \le s \le n - 1$, then there exists an $s$-ocycle on the set of permutations of $[n]$.
	\end{theorem}
	
	Definition \ref{def:socycle} is further justified by the following theorem showing the connection between graph $s$-overlap cycles and the clustered graph of $s$-overlapping permutations. The proof is a generalization of the proof of Theorem \ref{thm:perm graph bij} on gucycles for permutation graphs. It leverages Lemma \ref{lem:isoperm} to construct graph $s$-overlap cycles for permutation graphs given Euler tours of the clustered graph of $s$-overlapping permutations.

	\begin{theorem}\label{thm:sgobijn}For every $n \ge 3$ there is a bijection between Euler tours of the clustered graph of $s$-overlapping permutations and graph $s$-overlap cycles for permutation graphs that preserves the order in which $S_n$ is covered.\end{theorem}
	
	\begin{proof}By Lemma \ref{lem:isoperm} we need only show a bijection between Euler tours of $P_{n,s}$ and $s$-gocycles for permutation graphs on $[n]$ that preserves the order in which $S_n$ is covered.
		
		An Euler tour lists the permutation graphs on $[n]$ in an order $p_1,p_2,\ldots, p_{n!}$ such that consecutive permutation graphs overlap by $s$. Let $q_1 = p_1$, and for all $2 \le k \le n!$ let $q_k$ be the gluing of $(q_{k-1},p_k)$, which is possible by Lemma \ref{lem:gluing}. Then $q_{n!}$ is an $s$-goword $H$, which has length $n!(n-s)+s$, and since $(p_{n!},p_1)$ overlaps by $s$, we have that $(H,H)$ overlaps by $s$. Since the gluing operation does not create edges at distance greater than $s$, and $n \ge 3$, $H$ has no edges with one vertex in $\set{1,2,\ldots, s}$ and the other in $\set{n!(n-s)+1, n!(n-s)+2, \ldots, n!(n-s)+s}$. For every $n \ge 3$, with $m:=n!(n-s)+s = (n!+1)(n-s)+n$, we have $m-s = n!(n-s) \ge n! > 2(n-1)$, so by Lemma \ref{lem:scyclicgluing} there exists a cyclic gluing $J$ of $H$ which has the same $n$-windows starting at positions that are $1 \mod{n-s}$ in the same order. Therefore $J$ is an $s$-gocycle covering $S_n$ in the same order as the Euler tour.
		
		Given an $s$-gocycle $H$, the corresponding Euler tour of $P_{n,s}$  is $(p_1,p_2, \ldots, p_{n!})$, where for all $i \in [n!]$, $p_i := H[1+(i-1)(n-s), 2+(i-1)(n-s), \ldots, n+(i-1)(n-s)]$. That is, it is obtained by taking the relevant $n$-windows of $H$ in order.
	\end{proof}
	
	Analogously to gucycles for permutation graphs and universal cycles for permutations in Section \ref{subsec:permbijn}, we do not know whether there could be some Euler tours of the clustered graph of $s$-overlapping permutations that do not correspond to $s$-ocycles for permutations (although every Euler tour corresponds to an $s$-oword for permutations); see Remark \ref{rem:tourtoword}. However, Theorems \ref{thm:HHmod} and \ref{thm:sgobijn} together imply the existence of $s$-gocycles for permutation graphs on vertex set $[n]$ under suitable conditions on $s$ and $n$.
	
	\begin{corollary}
		Let $n, s \in \N$ with $n \ge 2$. If either (1) $1 \le s < n/2$, or (2) $\gcd(s, n) = 1$ with
		$n/2 \le s \le n - 1$, then there exists an $s$-gocycle for the permutation graphs on vertex set $[n]$.
	\end{corollary}
	
	\begin{proof}
		By Theorem \ref{thm:HHmod}, these conditions guarantee the existence of an $s$-overlap cycle for permutations. Given this $s$-ocycle, its $n$-windows give an ordering of $S_n$ such that every consecutive pair overlaps by $s$, which exactly corresponds to an Euler tour of the clustered graph of $s$-overlapping permutations. By Theorem \ref{thm:sgobijn}, the Euler tour of the clustered graph of $s$-overlapping permutations corresponds to an $s$-gocycle for the permutation graphs on vertex set $[n]$.
	\end{proof}
	
	An example of an $s$-gocycle is shown in Figure \ref{fig:sgocycle}.

	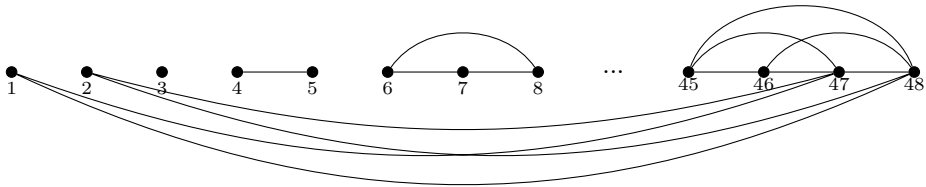
\begin{figure}[h!]
		\centering
		
		\tikzstyle{every node}=[circle, draw, fill=black, inner sep=0pt, minimum width=4pt]
		\begin{tikzpicture}
			
			\node[label={[shift={(0,-0.4)}]\scriptsize{1}}] (1) {};
			\node[label={[shift={(0,-0.4)}]\scriptsize{2}}] (2) [right of=1] {};
			\node[label={[shift={(0,-0.4)}]\scriptsize{3}}] (3) [right of=2] {};
			\node[label={[shift={(0,-0.4)}]\scriptsize{4}}] (4) [right of=3] {};
			\node[label={[shift={(0,-0.4)}]\scriptsize{5}}] (5) [right of=4] {};
			\node[label={[shift={(0,-0.4)}]\scriptsize{6}}] (6) [right of=5] {};
			\node[label={[shift={(0,-0.4)}]\scriptsize{7}}] (7) [right of=6] {};
			\node[label={[shift={(0,-0.4)}]\scriptsize{8}}] (8) [right of=7] {};
			
			\tikzstyle{every node}=[]
			\node[] [right of=8] (d) {...};
			
			\tikzstyle{every node}=[circle, draw, fill=black, inner sep=0pt, minimum width=4pt]
			\node[label={[shift={(0,-0.4)}]\scriptsize{45}}] (9) [right of=d] {};
			\node[label={[shift={(0,-0.4)}]\scriptsize{46}}] (10) [right of=9] {};
			\node[label={[shift={(0,-0.4)}]\scriptsize{47}}] (11) [right of=10] {};
			\node[label={[shift={(0,-0.4)}]\scriptsize{48}}] (12) [right of=11] {};
			\draw
			(4) to (5)
			(6) to (7) to (8)
			(9) to (10) to (11) to (12)
			
			(8) to[bend left=-55] (6)
			(11) to[bend left=-55] (9)
			(12) to[bend left=-55] (10)
			(12) to[bend left=-70] (9)
			
			(11) to[bend left=20] (1)
			(11) to[bend left=15] (2)
			(12) to[bend left=25] (1)
			(12) to[bend left=20] (2);

		\end{tikzpicture}   
		
		\caption{The beginning and end of the $2$-gocycle for permutation graphs on $[4]$ corresponding to the $2$-overlap cycle $012436543645362401324675689568732410425380756432$ from Example \ref{example:socycle}.}
		\label{fig:sgocycle}
	\end{figure}
	
	\subsection{Other types of gupwords and shortened universal words for permutations}\label{subsec:permfuture}
	
	The gupcycles for permutation graphs in Section \ref{subsec:permgupcycle} all are derived from the idea of using incomparable elements at distance $n-1$ to shorten universal words for permutations, from \cite{KPV19}. In this section we show examples of gupwords for permutation graphs outside of this framework, which suggest new directions for future research. The first two examples also demonstrate new ways to shorten universal words for permutations beyond those considered in \cite{KPV19}, which suggest new research questions on universal words.

	\begin{example}\label{ex:gupword1} Figure \ref{fig:gupword2} shows a new kind of gupword for permutation graphs on $[3]$. Here only one pair of twin edges, instead of a whole cycle of pairs of twin edges, has been compressed in the arc digraph for permutation graphs. As a result there is only one diamond edge in the gupword, and the gupword cannot be cyclically glued to form a gupcycle. This gupword naturally corresponds (in the sense of covering the permutations and their permutation graphs in the same partial order) to a new form of shortened universal word for $3$-permutations, $2134312$.
		\begin{figure}[h!]
			\centering
			\tikzstyle{every node}=[circle, draw, fill=black, inner sep=0pt, minimum width=4pt]
			
			\begin{tikzpicture}[scale=0.4]
				
				\node[label={[shift={(0,-0.4)}]\scriptsize{1}}] (1) {};
				\node[label={[shift={(0,-0.4)}]\scriptsize{2}}] (2) [right of=1] {};
				\node[label={[shift={(0,-0.4)}]\scriptsize{3}}] (3) [right of=2] {};
				\node[label={[shift={(0,-0.4)}]\scriptsize{4}}] (4) [right of=3] {};
				\node[label={[shift={(0,-0.4)}]\scriptsize{5}}] (5) [right of=4] {};
				\node[label={[shift={(0,-0.4)}]\scriptsize{6}}] (6) [right of=5] {};
				\node[label={[shift={(0,-0.4)}]\scriptsize{7}}] (7) [right of=6] {};
				
				\path[every node/.style]
				(1) edge node {} (2)    
				(4) edge node {} (5)
				(5) edge node {} (6);
				
				\draw[bend left=55]
				(5) to (7)
				(4) to (6);
				
				\draw [style=dashed, bend left=55] 
				(3) to (5);
				
			\end{tikzpicture}
			\caption{A gupword for permutation graphs on $[3]$ corresponding to the shortened universal word for permutations $2134312$.}
			\label{fig:gupword2}
		\end{figure}
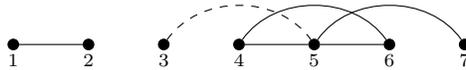
	\end{example}

	\begin{example}Figure \ref{fig:gupword1} shows a gupword for permutation graphs on $[3]$, which is different from the gupcycles considered in Section \ref{subsec:permgupcycle} and from Example \ref{ex:gupword1} in that it has a diamond edge at distance $1 \ne n-1$. This gupword also naturally corresponds to an analogously new kind of shortened universal word for $3$-permutations, $132231$, which uses incomparable elements at distance $1 \ne n-1$. It is an open question whether there exists a gupcycle with a diamond edge at distance less than $n-1$.
		\begin{figure}[h!]
			\centering
			\tikzstyle{every node}=[circle, draw, fill=black, inner sep=0pt, minimum width=4pt]
			
			\begin{tikzpicture}[scale=0.4]
				
				\node[label={[shift={(0,-0.4)}]\scriptsize{1}}] (1) {};
				\node[label={[shift={(0,-0.4)}]\scriptsize{2}}] (2) [right of=1] {};
				\node[label={[shift={(0,-0.4)}]\scriptsize{3}}] (3) [right of=2] {};
				\node[label={[shift={(0,-0.4)}]\scriptsize{4}}] (4) [right of=3] {};
				\node[label={[shift={(0,-0.4)}]\scriptsize{5}}] (5) [right of=4] {};
				\node[label={[shift={(0,-0.4)}]\scriptsize{6}}] (6) [right of=5] {};
				
				\path[every node/.style]
				(2) edge node {} (3)    
				(5) edge node {} (6);
				
				\draw[bend left=55]
				(2) to (4)
				(4) to (6);
				
				\draw [style=dashed] 
				(3) to (4);
				
			\end{tikzpicture}
			\caption{A gupword for permutation graphs on $[3]$ corresponding to the shortened universal word for permutations $132231$.}
			\label{fig:gupword1}
		\end{figure}
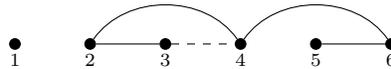
	\end{example}

	In \cite{KPV19}, two different methods were shown for shortening universal words for permutations. The first, using incomparable elements at distance $n-1$, is the method we used to show the existence of gupcycles for permutation graphs in Section \ref{subsec:permgupcycle}. The second method uses $\diam$ as a ``do not know'' symbol or $\diam_{\set{a,b}}$ to represent having two possible values, $a$ and $b$. The following example shows promise for translating this idea to the graph universal word setting. Gupwords and gupcycles for permutation graphs using a $\diamond$ symbol representing only two of the symbols could be explored in the future.
	
	\begin{example} Figure \ref{fig:restrictedgupword} shows two examples of universal words for $3$-permutations shortened using $\diam_{\set{a,b}}$, along with their corresponding gupwords.
		
		\begin{figure}[h!]
			\centering
			
			\begin{multicols}{2}
				\tikzstyle{every node}=[circle, draw, fill=black, inner sep=0pt, minimum width=4pt]
				
				\begin{tikzpicture}[scale=0.4]
					
					\node[label={[shift={(0,-0.4)}]\scriptsize{1}}] (1) {};
					\node[label={[shift={(0,-0.4)}]\scriptsize{2}}] (2) [right of=1] {};
					\node[label={[shift={(0,-0.4)}]\scriptsize{3}}] (3) [right of=2] {};
					\node[label={[shift={(0,-0.4)}]\scriptsize{4}}] (4) [right of=3] {};
					\node[label={[shift={(0,-0.4)}]\scriptsize{5}}] (5) [right of=4] {};
					\node[label={[shift={(0,-0.4)}]\scriptsize{6}}] (6) [right of=5] {};
					\node[label={[shift={(0,-0.4)}]\scriptsize{7}}] (7) [right of=6] {};
					
					\path[every node/.style]
					(3) edge node {} (4)    
					(4) edge node {} (5)    
					(6) edge node {} (7);
					
					\draw[bend left=55]
					(3) to (5)
					(5) to (7);
					
					\draw [style=dashed, bend left=55] 
					(1) to (3);
					
					\draw [style=dashed] 
					(1) to (2);
					
				\end{tikzpicture}
				
				\columnbreak
				
				\tikzstyle{every node}=[circle, draw, fill=black, inner sep=0pt, minimum width=4pt]
				\begin{tikzpicture}[scale=0.4]
					
					\node[label={[shift={(0,-0.4)}]\scriptsize{1}}] (1) {};
					\node[label={[shift={(0,-0.4)}]\scriptsize{2}}] (2) [right of=1] {};
					\node[label={[shift={(0,-0.4)}]\scriptsize{3}}] (3) [right of=2] {};
					\node[label={[shift={(0,-0.4)}]\scriptsize{4}}] (4) [right of=3] {};
					\node[label={[shift={(0,-0.4)}]\scriptsize{5}}] (5) [right of=4] {};
					\node[label={[shift={(0,-0.4)}]\scriptsize{6}}] (6) [right of=5] {};
					\node[label={[shift={(0,-0.4)}]\scriptsize{7}}] (7) [right of=6] {};
					
					\path[every node/.style]
					(3) edge node {} (4)    
					(4) edge node {} (5)    
					(6) edge node {} (7);
					
					\draw[bend left=55]
					(3) to (5)
					(4) to (6)
					(5) to (7);
					
					\draw [style=dashed] 
					(1) to (2);
					
				\end{tikzpicture}
			\end{multicols}
			
			\caption{Gupwords with restricted $\diam$'s: $\mathord{\diamond}_{\set{1,5}}243241$  (dashed edges must be both included or both excluded) on the left and $\mathord{\diamond}_{\set{1,3}}254231$ on the right.}
			\label{fig:restrictedgupword}
		\end{figure}
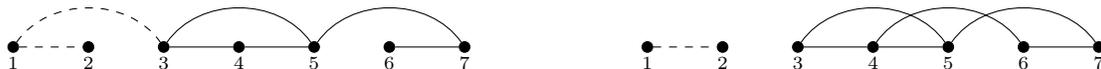
		
	\end{example}
	
	\section{Open Problems}\label{sec:open}
	
	Many interesting questions remain open. Most broadly, other families of combinatorial objects could be explored within the framework of graph universal (partial) cycles. Which other classes of $n$-vertex graphs, or of other combinatorial objects, can be represented by a gucycle or gupcycle?
	
	On the topic of gupcycles for labeled graphs on $[n]$, which gupcycle lengths are possible? Theorem \ref{thm:guplengths} ensures two intervals of lengths can always be achieved. The remaining question, Question \ref{qu:labeled}, is whether the intermediate lengths can be achieved for all $n \ge 5$. Further refining this question, it would be interesting to know how many gupcycles for labeled graphs on $[n]$ exist for each possible gupcycle length. For each length $\ell$ there may be multiple partially compressed digraphs with $\ell$ edges, and, within each, one could apply the BEST Theorem (see \cite{AB51,TS41}) to find the number of Euler tours or gupcycles.
	
	Conjecture \ref{conj:unlabeled4} on the existence of gucycles for $n$-vertex unlabeled graphs for all $n \ge 4$ remains open. Since every such gucycle can be split to construct a guword for $n$-vertex unlabeled graphs, a possibly easier question is whether guwords for $n$-vertex unlabeled graphs exist for all $n \ge 3$. Our results together with those in \cite{BKS10} imply that such guwords exist for $n \in \set{3,4,5}$. Gupcycles and gupwords for $n$-vertex unlabeled graphs could also be explored further; we have shown only that one exists.
	
	An open question on gupcycles for $n$-vertex threshold graphs is whether they exist for any $n \notin \set{5,9}$. One way to address the question is via Proposition \ref{prop:thresholdgupcycle}, by showing the existence of universal partial cycles for $\set{0,1}^{n-1}$ for other values of $n$ (the only possible values of $n$ are of the form $2^k+1$ \cite{FGKMM21}). Alternatively, do any gucycles and gupcycles for $n$-vertex threshold graphs exist that represent some of the threshold graphs in an order that does not match the usual binary representation?
	
	Turning our attention to permutations, do gucycles exist for subsets of the $n$-vertex permutation graphs that correspond to other families of combinatorial objects? For example, universal cycles for pattern-avoiding permutations were studied in \cite{AW09}; are there gucycles for pattern-avoiding permutation graphs? Extending our work, can $s$-gocycles be shortened using diamond edges? In addition, our new definition of $s$-overlap cycles for permutations may be of interest outside of the study of graph universal cycles. Section \ref{subsec:permfuture} shows a few different kinds of gupwords for permutation graphs along with new kinds of shortened universal words for permutations that could be explored in the future.

	Finally, graph universal cycles for $n$-vertex labeled trees were shown to exist for all $n\ge 3$ in \cite{BKS10}, but no meaningful correspondence between these gucycles and universal cycles is known. In particular, interpreting the $(n-2)$-windows of a De Bruijn cycle for $[n]^{n-2}$ as Pr\"{u}fer codes yields a universal cycle representation of the $n$-vertex labeled trees, which does not naturally give rise to a graph universal cycle for $n$-vertex labeled trees because overlaps in the Pr\"{u}fer codes do not correspond to overlaps in the labeled trees as ordered graphs. Other codes for $n$-vertex labeled trees have been studied, e.g. in \cite{DM01}. Is there a code for $n$-vertex labeled trees such that the De Bruijn cycles for $[n]^{n-2}$ can be used to produce gucycles and vice versa?
	
	\section{Acknowledgments}
	
	This research was supported by NSF DMS-1839918. We thank Dylan Fillmore, Bennet Goeckner, Xinsen Hong, Bernard Lidick\'{y}, Kirin Martin, Daniel McGinnis, and Jan van den Heuvel for helpful discussions.
	
	\bibliographystyle{plainurl}
	\bibliography{GucyclesBibliography}
\end{document}